\documentclass[11pt, english,a4paper]{amsart}

\usepackage[utf8]{inputenc}
\usepackage[T1]{fontenc}
\usepackage{color}
\usepackage{lmodern}
\usepackage{amsthm}
\usepackage[headings]{fullpage}
\usepackage{amssymb}
\usepackage{enumerate}
\usepackage{kantlipsum} 
\usepackage{tikz-cd}
\usepackage{stmaryrd}
\usepackage{tikz-cd}
\usepackage{hyperref}
\usepackage{mathrsfs}
\usepackage{bbold}

\usepackage[backend=biber, style=alphabetic, sorting=nyt]{biblatex}
\renewbibmacro{in:}{}

\calclayout

\newtheorem{theo}{Theorem}
\newtheorem{coro}[theo]{Corollary}
\newtheorem{lemm}[theo]{Lemma}

\newtheorem{prop}[theo]{Proposition}
\newtheorem{defi}[theo]{Definition}
\newtheorem{rema}[theo]{Remark}

\let\cal\mathcal

\let\hat\widehat
\let\tilde\widetilde
\let\phi\varphi

\let\epsilon\varepsilon

\def\Q{{\bf Q}} 
\def\Z{{\bf Z}}
\def\C{{\bf C}}
\def\N{{\bf N}}

\def\A{{\bf A}}

\def\B{{\bf B}}
\def\O{{\cal O}}
\def\G{{\cal G}}

\def\D{{\bf D}}

\def\id{{\mathrm{id}}}

\def\Mat{{\mathrm{Mat}}}

\def\coker{{\mathrm{coker}}}

\def\Zp{{\Z_p}}
\def\Qp{{\Q_p}}
\def\Cp{{\C_p}}
\def\OCp{{\O_{\Cp}}}

\def\At{{\tilde{\bf{A}}}}
\def\Atplus{{\tilde{\bf{A}}^+}}
\def\Bt{{\tilde{\bf{B}}}}
\def\Btplus{{\tilde{\bf{B}}^+}}

\def\Gal{{\mathrm{Gal}}}

\def\GL{{\mathrm{GL}}}

\def\cycl{{\mathrm{cycl}}}

\newcommand{\cont}{\mathrm{cont}}

\newcommand{\im}{\mathrm{im}\>}
\newcommand{\tgamma}{\tilde{\gamma}}
\newcommand{\ttau}{\tilde{\tau}}

\newcommand{\cE}{\mathcal{E}}

\newcommand{\Dt}{\tilde{\D}}

\author{Anand Chitrao}
\address{Ulsan National Institute of Science and Technology \\
Ulsan}
\email{anand@unist.ac.kr; anandchitrao@gmail.com}
\urladdr{https://sites.google.com/view/anandchitrao}

\author{Aditya Karnataki}
\address{Chennai Mathematical Institute\\
Chennai}
\email{adityack@cmi.ac.in; karnatakiaditya@gmail.com}
\urladdr{https://adityakarnataki.github.io/}

\author{Jishnu Ray}
\address{Harish-Chandra Research Institute \\
Prayagraj}
\email{jishnuray@hri.res.in; jishnuray1992@gmail.com}
\urladdr{https://sites.google.com/site/rayjishnu1992}

\date{\today}

\title{Explicit isomorphisms for a Herr-type complex over a metabelian extension}

\begin{document}

{\renewcommand\thefootnote{}}
{\renewcommand\thefootnote{}}

\begin{abstract}
    Let $S$ be a Banach algebra over $\Qp$ whose residue fields are finite extensions of $\Qp$. Given an arithmetic family $V$ of Galois representations, i.e., a finite free $S$-module $V$ with a continuous action of the absolute Galois group of a $p$-adic number field, we construct a complex associated to $V$ over false-Tate extensions and construct explicit isomorphisms between its cohomology and the Galois cohomology. This recovers earlier results by Tavares Ribeiro when $S$ is a finite extension of $\Qp$.  
\end{abstract}

\maketitle

\section{Introduction}
\subsection{Notations:} Let $p$ be an odd prime. Let $K$ be a finite extension of $\Qp$, with ring of integers $\O_K$ and residue field $k$ and let $\pi$ be a uniformizer of $K$. We let $\overline{K}$ be an algebraic closure of $K$ and let $\C_p$ be the $p$-adic completion of $\overline{K}$ with ring of integers $\OCp$. Let $v_p$ be the $p$-adic valuation on $\Cp$ such that $v_p(p)=1$. { Let $v_{\C_{p}^{\flat}}$ be the valuation on $\C_{p}^{\flat}$ obtained using the isomorphism of monoids $ \C_{p}^{\flat} \simeq \underset{x \mapsto x^p}{\varprojlim}\Cp$ by $v_{\C_{p}^{\flat}}(x) = v_p(x^{\#})$}. Let $(\pi_n)$ be a sequence of elements of $\overline{K}$, such that $\pi_0 =\pi$ and $\pi_{n+1}^p = \pi_n$. We let $K_n = K(\pi_n)$ and $K_\infty = \bigcup_{n \geq 1}K_n$. Let also $\epsilon_1$ be a primitive $p$-th root of unity and $(\epsilon_n)_{n \in \N}$ be a compatible sequence of $p^n$-th roots of unity, which means that $\epsilon_{n+1}^p=\epsilon_n$ and let $K_{\mathrm{cycl}} = \bigcup_{n \geq 0}K(\epsilon_n)$ be the cyclotomic extension of $K$. Let $L:=K_\infty \cdot K_{\mathrm{cycl}}$ be the Galois closure of $K_\infty/K$, and let 
\[
G_\infty = \Gal(L/K), \quad H_\infty = \G_L = \Gal(\overline{K}/L), \quad \Gamma_K = \Gal(L/K_\infty).
\] Note that we can identify $\Gamma_K$ with $\Gal(K_{\mathrm{cycl}}/(K_\infty \cap K_{\mathrm{cycl}}))$ and so with an open subgroup of $\Z_p^\times$. 

For $g \in \G_K$ and for $n \in \N$, there exists a unique element $c_n(g) \in \Z/p^n\Z$ such that $g(\pi_n) = \epsilon_n^{c_n(g)}\pi_n$. Since $c_{n+1}(g) \equiv c_n(g) \mod p^n$, the sequence $(c_n(g))$ defines an element $c(g)$ of $\Zp$. The map $g \mapsto c(g)$ is actually a (continuous) $1$-cocycle of $\G_K$ to {$\Zp(1)$} (the rank $1$ module over $\Zp$ with $\G_K$ acting by the $p$-adic cyclotomic character $\chi_{\mathrm{cycl}}$), such that $c^{-1}(0) = \Gal(\overline{K}/K_{\infty})$, and satisfies for $g,h \in \G_K$~:
$$c(gh) = c(g)+\chi_{\mathrm{cycl}}(g)c(h).$$

This means that if $\Zp \rtimes \Z_p^{\times}$ is the semi-direct product of $\Zp$ with $\Z_p^{\times}$, where $\Z_p^{\times}$ acts on $\Zp$ by multiplication, then the map $g \in \G_K \mapsto (c(g),\chi_{\mathrm{cycl}}(g)) \in \Zp \rtimes \Z_p^{\times}$ is a morphism of groups with kernel $H_{\infty}$. The cocycle $c$ factors through $H_{\infty}$, and so defines a cocycle that we will still denote by $c~: G_\infty \to \Zp$, which is called Kummer's cocycle attached to $K_\infty/K$.

We let $\gamma$ be a topological generator of $\Gal(L / K_{\infty})$. We let $\tau$ be a topological generator of $\Gal(L/K_{\mathrm{cycl}})$ such that $c(\tau)=1$ (this is exactly the element corresponding to $(1,1)$ \textit{via} the isomorphism $g \in G_{\infty} \mapsto (c(g),\chi_{\mathrm{cycl}}(g)) \in \Zp \rtimes \Z_p^\times$). The relation between $\tau$ and $\Gamma_K$ is given by
\begin{eqnarray}\label{The monodromy identity}
    g\tau g^{-1} = \tau^{\chi_{\mathrm{cycl}}(g)} \text{ for all } g \in \Gamma_K.
\end{eqnarray}

We also let $H_K = \Gal(\overline{K}/K_{\mathrm{cycl}})$ and $H_{\tau,K} = \Gal(\overline{K}/K_\infty)$. If $A$ is an algebra endowed with an action of $\G_K$, we let $A_K = A^{H_K}$ and $A_{\tau,K} = A^{H_{\tau,K}}$.

We fix a continuous section to the canonical projection $\G_K \to G_{\infty}$ as follows. Choose an arbitrary lift $\tau' \in H_K$ of $\tau$. Then, the procyclic subgroup generated by $\tau'$ must have a factor isomorphic to $\Zp$ surjecting onto the copy of $\Zp$ generated by $\tau$ in $G_{\infty}$. Let $\ttau$ be the element in this $\Zp$ factor mapping to $\tau$. So we get a continuous group theoretic section to the canonical projection $H_K \to \langle \tau \rangle$. Next, choose a continuous section to the projection $H_{\tau, K} \to \langle \gamma \rangle$. Defining $\tilde{\gamma^a \tau^b} = \tilde{\gamma^a}\tilde{\tau^b}$ yields a continuous section to the surjection $\G_K \to G_{\infty}$. We note that for any $a \in \Zp$, $\tilde{\tau^a} = \ttau^a$.

{ If $\chi_n$ is a sequence of elements in $\Z_{> 0}$ converging to $\chi(\gamma)$ in $\Z^{\times}_p$, then for a topological $\Gal(L/K_{\cycl})$-module, the linear operators $1 + \tau + \cdots + \tau^{\chi_n - 1}$ converge to a linear operator $\dfrac{\tau^{\chi(\gamma)} - 1}{\tau - 1}$, which we denote by $\delta$. This operator is invertible and its definition is independent of the choice of $\chi_n$.}

The following diagram is a concise way of remembering these extensions and the Galois groups:
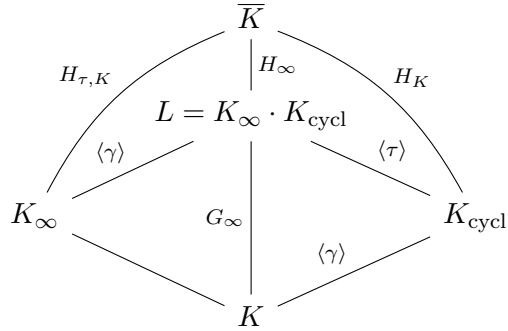
\[
    \begin{tikzcd}
        & \overline{K} & \\
        & L = K_\infty \cdot K_{\mathrm{cycl}} \ar[u, "H_\infty", no head, swap]& \\
        K_\infty \ar[ur, "\langle \gamma \rangle", no head]\ar[uur,"H_{\tau, K}",bend left=20,looseness= 1, no head] & & K_{\mathrm{cycl}} \ar[ul, "\langle \tau \rangle", no head, swap] \ar[uul, "H_K", bend right = 20, looseness = 1, no head, swap] \\
        & K \ar[ul, no head] \ar[ur, "\langle \gamma \rangle", no head] \ar[uu, "G_\infty", no head]&
    \end{tikzcd}
\]

Here are some conventions we adopt throughout this paper. If $G$ is a group and $M$ is a $G$-module, the class of a cocycle $c : G^i \to M$ in $H^i(G, M)$ is denoted by $[c]$. If $M_i$ for $i \in \{1, 2, \ldots, I\}$ and $N_j$ for $j \in \{1, 2, \ldots, J\}$ are modules with maps $a_{i, j} : N_j \to M_i$, then the map $\oplus_{j = 1}^{J} N_j \to \oplus_{i = 1}^{I} M_i$ will be denoted by the matrix $(a_{i, j})$ and applying this map should be thought of as multiplying the $I \times J$ matrix $(a_{i, j})$ with a $J \times 1$ column vector coming from the domain to get an $I \times 1$ column vector in the codomain.

\subsection{$(\varphi,\Gamma)$-modules} 
A central theme in $p$-adic Hodge theory is to classify $p$-adic representations of $\G_K$. To this end, Fontaine \cite{Fon90} introduced the category of $(\varphi, \Gamma)$-modules. Moreover, he constructed an equivalence $V \rightarrow D(V)$ between the category of $p$-adic representations of $\G_K$ and the category of \'{e}tale $(\varphi, \Gamma)$-modules 
over a local field $\B_K$, equipped with semi-linear, commuting actions of a Frobenius operator $\varphi$ and that of $\Gamma$, which is an open subgroup of $\Z_p^{\times}$.

Subsequently, $(\varphi, \Gamma)$-modules over variants of the ring $\B_K$ were constructed by Cherbonnier-Colmez \cite{CC98} and Kedlaya \cite{Ked04}. The category of \'etale $(\varphi, \Gamma)$-modules over any one of these rings is equivalent to the category of $p$-adic Galois representations. For an introduction, see the articles of Berger \cite{Ber04}, \cite{Ber11}. 
In \cite{herr1998cohomologie}, Herr defines a three-term complex of  $(\varphi, \Gamma)$-modules which computes the cohomology of the  Galois representation when the corresponding  $(\varphi, \Gamma)$-module is étale. 

Let $S$ be a $\Q_p$-Banach algebra such that for every $x$ in the maximal spectrum of $S$, $S/\mathfrak{m}_x$ is a finite extension of $\Q_p$. Let $V$ be a family of representations of $\G_K$ over $S$ as in section~\ref{Families of phi tau modules}. Associated to $V$ is a family of étale $(\varphi, \Gamma)$-modules $\D(V)$ (thanks to the work of Berger--Colmez~\cite{BC08} and Kedlaya--Liu~\cite{KL10}). 
Pottharst \cite{Pot13} showed that the same Herr's complex in the setting of families computes the Galois cohomology. 

The strategy of Pottharst is as follows. He first shows that the cohomology of $V$ as a representation of $H_K $ can be computed using the two-term complex
$\mathcal{C}_D^\bullet : \D(V) \xrightarrow{\varphi - 1} \D(V).
$
It is well-known that the cohomology of a $\Gamma$-module $M$ is computed using the two-term complex
$M \xrightarrow{\gamma - 1} M,$
where $\gamma$ is a generator of $\Gamma$.  One can then consider the $\gamma - 1$ map from the complex $\mathcal{C}_D^\bullet$ to itself. Using the Hochschild-Serre spectral sequence, he showed that the resulting total complex computes the Galois cohomology. We note that $H^2(\Gamma, V^{H_K}) = 0$. Therefore $H^1(\G_K, V)$ is sandwiched between $H^1(\Gamma, V^{H_K})$ and $H^1(H_K, V)^{\Gamma}$. Using this, one can produce an explicit isomorphism between $H^1(\G_K, V)$ and $H^1(\mathrm{Tot}(\mathcal{C}_D^\bullet \xrightarrow{\gamma - 1} \mathcal{C_D^\bullet}))$. The isomorphism between $H^2(\G_K, V)$ and $H^1(\Gamma, H^1(H_K, V))$ comes out for free from the Hochschild-Serre spectral sequence since $H^i(\Gamma, \_)$ and $H^i(H_K, \_)$ are zero for $i \geq 2$. It is easy to see that $H^1(\Gamma, H^1(H_K, V))$ is isomorphic to $H^2(\mathrm{Tot}(\mathcal{C}_D^\bullet \xrightarrow{\gamma - 1} \mathcal{C_D^\bullet}))$.

\subsection{$(\varphi,\tau)$-modules}

By the theory of Fontaine-Wintenberger, one notes that the absolute Galois group $H_{\tau, K}$ of the Kummer tower $K_{\infty}$ is isomorphic to that of a characteristic $p$ field. Using this observation, one may classify representations of $H_{\tau, K}$ using $\varphi$-modules. A naïve approach will be to hope to classify representations of $\G_K$ by first restricting to $H_{\tau, K}$, passing to the $\varphi$-module and then considering the  $\G_K/H_{\tau, K}$-action. This does not work since $H_{\tau, K}$ is not a normal subgroup of $\G_K$. One indeed considers the $\varphi$-module associated to the restriction of the Galois representation to $H_{\tau, K}$. The datum of the remaining $\G_K/H_{\tau, K}$-action can be obtained by studying the $G_{\infty}$-action on a base change of this $\varphi$-module. This gives rise to the notion of $(\varphi, \tau)$-modules. These were first studied by Tavares Ribeiro \cite{tavares2008phi} and Caruso~\cite{Car13} using slightly different approaches. (See section~\ref{(phi, tau)-section} for details.) In the particular case of semi-stable representations, these $(\varphi, \tau)$-modules
coincide with the notion of Breuil-Kisin modules and can thus be used to study
Galois deformation rings as in \cite{kisin2008potentially}, to classify semi-stable (integral) Galois representations as in \cite{Liu10}, and to study integral models of Shimura varieties as
in \cite{kisin2010integral}.

As in the case of classical $(\varphi, \Gamma)$-modules, one may ask if, for a representation $V$ of $\G_K$ defined over a finite extension of $\Q_p$, can one compute $H^i(\G_K, V)$ from a complex made out of the corresponding $(\varphi, \tau)$-module. Indeed, it was shown by Tavares Ribeiro in \cite{Rib11} that there is a four-term complex made out of a base change of the $(\varphi, \tau)$-module associated to $V$ which computes $H^i(\G_K, V)$. Subsequently, other papers \cite{Zha25, GZ24} also wrote down other three-term complexes made out of the corresponding $(\varphi, \tau)$-module all of which compute $H^i(\G_K, V)$. We note that \cite{GZ24} wrote down a three-term complex which computes the Galois cohomology for representations of $\G_K$ defined over finite extensions of $\Qp$ of which we were not aware while writing this article. It might be worthwhile to explore the consequences of their results in our relative setting of families, but we have not done so. In our paper, we use completely different techniques to work with representations defined over arbitrary $\Qp$-affinoid algebras as our focus is also on the explicit maps realizing the isomorphism of Galois cohomology with the cohomology of the complex stated in Theorem~\ref{thm:mainIntro}.

In this paper, we work with families $V$ of representations of $\G_K$, i.e., representations defined over general $\Qp$-affinoid algebras $S$. Our goal is to find a complex made out of the family of $(\varphi, \tau)$-modules associated to $V$ by \cite{KP23} which computes $H^i(\G_K, V)$. We indeed write down a complex and produce explicit maps between its cohomology groups and $H^i(\G_K, V)$. Furthermore, our complex specializes to that of Tavares Ribeiro when $S$ is a finite extension of $\Qp$.

\begin{theo}[see Theorem \ref{Overconvergent theorem}]\label{thm:mainIntro}
       
       Let $V$ be a family of representations of $\G_K$ and $\D^{\dagger, r}_{\tau, K}(V)$ the associated $(\varphi, \tau)$-module as in section~\ref{Families of phi tau modules} (Theorem~\ref{theo ana families}). For $r \geq r_0$, let $\D^{r}_L = S\hat{\otimes}\Bt^{\dagger, r}_L\otimes_{S \hat{\otimes}\B_{\tau, K}^{\dagger, r}} D_{\tau, K}^{\dagger, r}(V)$ (for a definition of these period rings, see section~\ref{Period rings}). Then, the cohomology of the complex
        {\small
        \[
            0 \to \D^{r}_L \xrightarrow{d_0 = \begin{pmatrix}\varphi - 1 \\ \gamma - 1 \\ \tau - 1\end{pmatrix}} \D^{pr}_L \oplus \D^{r}_L \oplus \D^{r}_L \xrightarrow{d_1 = \begin{pmatrix}\gamma - 1 & 1 - \varphi & 0 \\ \tau - 1 & 0 & 1 - \varphi \\ 0 & \tau - 1 & 1 - \delta^{-1}\gamma\end{pmatrix}} \D^{pr}_L \oplus \D^{pr}_L \oplus \D^{r}_L \xrightarrow{d_2 = (\tau - 1, 1 - \delta^{-1}\gamma, \varphi - 1)} \D^{pr}_L \to 0
        \]
        }
        is isomorphic to the Galois cohomology of $V$, where $\delta = \frac{\tau^{\chi(\gamma)} - 1}{\tau - 1}$.
    \end{theo}

First, we make a few remarks. 
\begin{rema}
    The proof of Tavares Ribeiro (in the setting where $S$ is a finite extension of $\Qp$) starts with an embedding of the representation $V$ into a representation $\mathrm{Ind}\> V$ which has no cohomologies in positive degree. He then uses a dimension-shifting argument with the $(\varphi, \tau)$-modules associated with $V$ and $\mathrm{Ind}\> V$. This approach seems difficult in our setting since it seems difficult to adapt it to the theory of infinite dimensional families of representations of $\G_K$. Let alone infinite dimensional representations of $\G_K$, the construction of families of $(\varphi, \tau)$-modules corresponding to finite dimensional families of representations is non-trivial (see \cite[Theorem 48]{KP23}). 

\vspace{.2cm}

    Contrary to Pottharst's approach in the setting of families of $(\varphi, \Gamma)$-modules, the group $G_{\infty}$ has cohomolgies in degree $2$. This means that we have to work with $H^2(G_{\infty}, V^{H_{\infty}})$ even to prove that $H^1$ of the complex above is isomorphic to $H^1(\G_K, V)$. Moreover, the isomorphism for $H^2(\G_K, V)$ does not come out for free from the Hochschild-Serre spectral sequence, since again $H^2(G_{\infty}, V^{H_{\infty}})$ is possibly non-zero.

\vspace{.2cm}

\end{rema}

\noindent{\textit{Idea of the proof}:} 
    Analogous to Pottharst's method, we show that the cohomology of a family of representations of $\G_K$, when restricted to $H_{\infty}$, can be computed using the sequence $\D^{r} \xrightarrow{\varphi - 1} \D^{pr}$ (see Proposition~\ref{Hinfty cohomology}). Next, we show that if $M$ is a $G_{\infty}$-module, then the cohomology of $M$ can be computed using the three-term complex $M \xrightarrow{f_M} M \oplus M \xrightarrow{g_M}M$ (see Proposition~\ref{Explicit Ginfty cohomology}). 
    These explicit maps are essential in the proof of Theorem~\ref{thm:mainIntro}. Finally, we put these two together and consider the following complex
    \[
        \begin{tikzcd}
            \D^{pr} \ar[r, "f_{\D^{pr}}"] & \D^{pr}\oplus \D^{pr} \ar[r, "g_{\D^{pr}}"] & \D^{pr} \\
            \D^{r} \ar[r, "f_{\D^{r}}"]\ar[u, "\varphi - 1"] & \D^{r} \oplus \D^{r} \ar[r, "g_{\D^{r}}"] \ar[u, "(\varphi - 1)\oplus (\varphi - 1)"] & \D^{r} \ar[u, "\varphi - 1"].
        \end{tikzcd}
    \]
    Taking the total complex now yields the complex in Theorem~\ref{thm:mainIntro}.

    It remains to show that the cohomology of this complex is isomorphic to $H^i(\G_K, V)$. The case of $i = 0$ is easy and can be done separately. Next we treat the $i = 1$ and $2$ cases. The degeneration of the Hochschild-Serre spectral sequence yields the following exact sequence (see \cite[Pg. 71]{DHW12})
    \begin{eqnarray*}
        H^1(G_{\infty}, V^{H_{\infty}}) \hookrightarrow H^1(\G_K, V) \xrightarrow{} H^1(H_{\infty}, V)^{G_{\infty}} \xrightarrow{} H^2(G_{\infty}, V^{H_{\infty}})
        \xrightarrow{} H^2(\G_K, V) \twoheadrightarrow H^1(G_{\infty}, H^1(H_{\infty}, V)).
    \end{eqnarray*}
    We construct the following analogous exact sequence for $(\varphi, \tau)$-modules (see Proposition~\ref{Substitute for 6 term HS})
    \begin{eqnarray*}
         \dfrac{\ker g_{(\D^{r})^{\varphi = 1}}}{\im f_{(\D^{r})^{\varphi = 1}}} \hookrightarrow \dfrac{\ker d_1}{\im d_0} \to \left(\dfrac{\D^{pr}}{(\varphi - 1)\D^{r}}\right)^{\tau = 1, \gamma = 1} \to \dfrac{(\D^{r})^{\varphi = 1}}{(\tau - 1, 1 - \delta^{-1}\gamma)} \to \dfrac{\ker d_2}{\im d_1} \twoheadrightarrow \dfrac{\ker g_{(\D^{pr}/(\varphi - 1)\D^{r})}}{\im f_{(\D^{pr}/(\varphi - 1)\D^{r})}}.
    \end{eqnarray*}
    Then, using Propositions~\ref{Hinfty cohomology}, \ref{Explicit Ginfty cohomology}, we see that the first, third, fourth and sixth terms in the two sequences displayed above are isomorphic. The technical heart of this paper is to construct a homomorphism between the second terms as well as a homomorphism between the fifth terms in these two exact sequences which make the resulting diagram commute. Applying the five lemma then shows that these homomorphisms are indeed isomorphisms.
    
\vspace{.2cm}

    \textit{Technicalities:}
    We emphasize that it takes some effort even to define a map $h_2 : H^2(\G_K, V) \to \dfrac{\ker d_2}{\im d_1}$. Given a class $\alpha \in H^2(\G_K, V)$, we choose lifts $x, y$ to $\D^{pr}$ of the two coordinates of the image of $\alpha$ under the composition of $H^2(\G_K, V) \to H^1(G_{\infty}, H^1(H_{\infty}, V)) \simeq \dfrac{\ker g_{(\D^{pr}/(\varphi - 1)\D^{r})}}{\im f_{(\D^{pr}/(\varphi - 1)\D^{r})}}$. We show that these lifts determine a representative $c$ of $\alpha$ almost completely (see Lemmas~\ref{Existence of a c}, \ref{Second coordinate only matters mod H}, and \ref{lem:c(h,g)}). This is sufficient for us to prove

     \begin{prop}
        Let $s^{(pr)}$ be a section to the map $\varphi - 1 : (S \hat{\otimes}_{\Qp} \Bt^{\dagger, r}) \otimes_S V \to (S \hat{\otimes}_{\Qp} \Bt^{\dagger, pr}) \otimes_S V$ fixed at the end of Section~\ref{Cohomology of Hinfty modules}. Then, for $c : \G_K \times \G_K \to V$ obtained using Lemma~\ref{Existence of a c}, the element $z = c(\tgamma, \tgamma^{-1}\ttau\tgamma) - c(\ttau, \tgamma) - (\ttau - 1)s^{(pr)}x + \tgamma s^{(pr)}\left(\frac{\tau^{\chi(\gamma)^{-1}} - 1}{\tau - 1}y\right) - s^{(pr)}y \in (S \hat{\otimes}_{\Qp} \Bt^{\dagger, r}) \otimes_S V$ belongs to $\D^{r}$, i.e., it is $H_{\infty}$-fixed.
        Furthermore, the map
        \begin{eqnarray*}
            h_2 : H^2(\G_K, V) & \to & \frac{\ker d_2}{\im d_1} \\
            \alpha & \mapsto & \Big(-x, -y, -z\Big) \mod \im d_1
        \end{eqnarray*}
        is independent of the choices of $x$, $y$ and $c$, and is an $S$-linear homomorphism.
    \end{prop}

    \begin{rema}
        If we specialize Theorem~\ref{thm:mainIntro} to the case where $S$ is a finite extension of $\Qp$, we get another proof of the result of Tavares Ribeiro, \cite[Theorem 0.2]{Rib11}.
    \end{rema}

    The theory of cohomology of $(\varphi, \Gamma)$-modules via explicit complexes was established in the case of the cyclotomic tower by Herr~\cite{herr1998cohomologie}. Ruochuan Liu established fundamental properties of these complexes and the cohomology in \cite{Liu08}. Kedlaya-Pottharst-Xiao generalized these to the case of families in \cite{KPX01}. All of these results have found numerous applications in related areas of number theory. In a similar manner, it is thus desirable to establish foundational properties of the cohomology of families of $(\varphi, \tau)$-modules. These would then have similar applications in number theory; Theorem \ref{thm:mainIntro} should be considered as a first step in this direction.

\section{Preliminaries}
In this section, we collect some preliminary definitions and results that will be used throughout the paper.
\subsection{Period rings}\label{Period rings}
Recall the classical rings from, e.g., \cite{Ber04}
\[
\At = W(\C_p^{\flat}),  \quad\Atplus = W(\O_{\Cp}^{\flat}), \quad \Bt = \At[1/p] \quad \textrm{and } \Btplus = \Atplus[1/p].
\]
These rings are equipped with a Frobenius operator $\phi$, and with a $\G_\Qp$-action commuting with $\varphi$ which lifts the action of $\G_{\Qp}$ on $\C_p^{\flat}$. We endow these rings with the weak topology, for instance, this is the topology on $\At$ for which the monoid isomorphism $\At \simeq (\C_p^{\flat})^{\N}$ is a homeomorphism with the valuation topology on $\C_p^{\flat}$.

For $r > 0$, we define the subset of overconvergent elements of radius $r$ of $\Bt$, by
\[
    \Bt^{\dagger,r} := \left\{x = \sum_{n \gg -\infty}p^n[x_n] \textrm{ such that } \lim\limits_{n \to +\infty}v_{\C_p^{\flat}}(x_n)+\frac{pr}{p-1}n =+\infty \right\}
\]
and we let $\Bt^\dagger = \bigcup_{r > 0}\Bt^{\dagger,r}$ be the subset of all overconvergent elements of $\Bt$. For properties of these rings, see \cite{Ber04}. We provide $\Bt^{\dagger}$ with the subspace topology coming from $\Bt$.

In this paper, we need the ``$\tau$''-versions of these rings. See \cite[Section $1$]{KP23} for the definitions and properties of $\A_{\tau, K}$, $\B_{\tau, K}$, $\B_{\tau, K}^{\dagger, r}$ and $\B_{\tau, K}^{\dagger}$. In particular, note that $\B_{\tau, K} := \A_{\tau, K}[1/p]$, $\B_{\tau, K}^{\dagger, r} = \B_{\tau, K} \cap \Bt^{\dagger, r}$ and $\B_{\tau, K}^{\dagger} = \cup_{r > 0}\B_{\tau, K}^{\dagger, r}$. For a ring $R$, following \cite{Pot13}, the superscript $(\>, r)$ denotes $R$ or $R^{r}$. We put the subscript $L$ in order to define the $H_{\infty}$-fixed points of rings such as $\Bt_L := (\Bt)^{H_{\infty}}, \Bt^{\dagger}_L := (\Bt^{\dagger})^{H_{\infty}}$. These two rings get the subspace topology from $\Bt$ and $\Bt^{\dagger}$, respectively.
The Frobenius on $\Bt$ defines, by restriction, endomorphisms of $\A_{\tau,K}$ and $\B_{\tau,K}$.

To define the functor which associates a $(\varphi, \tau)$-module with a Galois representation, we need some more fields. Let $\widehat\B_{\tau, K}^{\mathrm{ur}}$ and $\widehat\B_{\tau, K}^{\mathrm{ur}, \dagger, r}$ be the $p$-adic
completions of the maximal unramified extensions of $\B_{\tau, K}$ and $\B_{\tau, K}^{\dagger, r}$, respectively. The former is denoted by $\widehat\cE^{\mathrm{ur}}$ in \cite[Section $2.1$]{GL20} (see also Section $2.5$ of loc. cit. for the overconvergent period rings).

\subsection{More on the theory of $(\varphi, \tau)$-modules}\label{(phi, tau)-section}
Caruso has defined $(\varphi, \tau)$-modules in \cite{Car13} for an odd prime $p$. His definition of $(\varphi, \tau)$-modules was modified by Gao and Liu in \cite[Definition 2.1.5]{GL20} to include the prime $p = 2$. These two definitions are the same if $p \neq 2$. Even if we assume that $p$ is odd, we work with the latter definition.

A $\phi$-module $D$ on $\B_{\tau,K}$ is a $\B_{\tau,K}$-vector space of dimension $d$, equipped with a $\varphi$-semilinear operator $\varphi_D$, such that the induced map 
$1 \otimes \varphi_D : \B_{\tau, K} \otimes_{\varphi, \B_{\tau, K}}D \to D$ is an isomorphism, and we say that it is étale if there exists a basis of $D$ in which $\Mat(\phi_D) \in \GL_d(\A_{\tau,K})$. Then, we may define $(\varphi, \tau)$-modules as follows.

\begin{defi}
\label{def phitau}
A $(\phi,\tau)$-module on $(\B_{\tau,K},\Bt_L)$ is a triple $(D,\phi_D,G)$, where
\begin{enumerate}
\item $(D,\phi_D)$ is a $\phi$-module on $\B_{\tau,K}$,
\item $G$ is \textit{the datum of} a continuous (for the weak topology) semilinear $G_\infty$-action on $M:=\Bt_L \otimes_{\B_{\tau,K}}D$ such that $G_{\infty}$ commutes with $\phi_M:=\phi_{\Bt_L}\otimes \phi_D$, i.e., for all $g \in G_\infty$, we have $g\phi_M = \phi_Mg$,
\item regarding $D$ as a sub-$\B_{\tau,K}$-module of $M$, $D \subset M^{H_{\tau,K}}$.
\end{enumerate}
We say that a $(\phi,\tau)$-module is étale if its underlying $\phi$-module over $\B_{\tau,K}$ is \'etale. Let the category of \'etale $(\varphi, \tau)$-modules over $(\B_{\tau, K}, \Bt_L)$ be denoted by $\mathrm{Mod}_{(\varphi, \tau)}^{\mathrm{\'et}}$.
\end{defi}

Let $\mathrm{Rep}_{\Qp}(\G_K)$ be the category of $\Qp$-linear, finite dimensional representations of $\G_K$ and suppose $V$ is in $\mathrm{Rep}_{\Qp}(\G_K)$. Then $(\widehat{\cE^{\mathrm{ur}}} \otimes_{\Qp} V)^{H_{\tau, K}}$ is a finite dimensional vector space over $\B_{\tau, K}$ with a $\varphi$-action induced from $\varphi \otimes \id$ on $\widehat{\cE^{\mathrm{ur}}}\otimes_{\Qp}V \subseteq \Bt \otimes_{\Qp}V$. The association $V \mapsto (\widehat{\cE^{\mathrm{ur}}} \otimes_{\Qp} V)^{H_{\tau, K}}$ defines a functor denoted by $D_{\tau, K}: \mathrm{Rep}_{\Qp}(\G_K) \to \mathrm{Mod}_{(\varphi, \tau)}^{\'et}$. Then, the following result was first proved by Caruso (see \cite[Th\'eor\`eme 1]{Car13}) and later extended to include $p = 2$ by Gao and Liu (see \cite[Proposition $2.17$]{GL20}).

\begin{theo}
    The functor $D_{\tau, K}: \mathrm{Rep}_{\Qp}(\G_K) \to \mathrm{Mod}_{(\varphi, \tau)}^{\'et}$ is an equivalence of categories.
\end{theo}

There is another version of the theory of $(\varphi, \tau)$-modules, namely the overconvergent ones. An \'etale $(\varphi, \tau)$-module over $(\B_{\tau, K}^{\dagger}, \Bt_{L}^{\dagger})$ is defined similarly, and the category of these modules is denoted by $\mathrm{Mod}_{(\varphi, \tau)}^{\dagger,  \'et}$. For $r > 0$, define $D_{\tau, K}^{\dagger, r}(V) = (\widehat{\B}_{\tau, K}^{\mathrm{ur}, \dagger, r} \otimes_{\Qp} V)^{H_{\tau, K}}$ and $D_{\tau, K}^{\dagger}(V) = \varinjlim_{r > 0} D_{\tau, K}^{\dagger, r}(V)$. We say that the $(\varphi, \tau)$-module $D_{\tau, K}(V)$ is overconvergent, if $D_{\tau, K}(V) \simeq \B_{\tau, K} \otimes_{\B_{\tau, K}^{\dagger}} D_{\tau, K}^{\dagger}(V)$. Then Gao-Liu \cite[Theorem $2.5.2$]{GL20} and Gao-Poyeton \cite[Theorem $1.1.2$]{GP21} prove the following:

\begin{theo}
    The functor $D_{\tau, K}^{\dagger} : \mathrm{Rep}_{\Qp}(\G_K) \to \mathrm{Mod}_{(\varphi, \tau)}^{\dagger, \'et}$ is an equivalence of categories.
\end{theo}

\subsection{Families of $(\varphi, \tau)$-modules}\label{Families of phi tau modules}
In this section, we recall some results regarding modules over $\Qp$-Banach algebras with a continuous $\G_K$-action.

Let $S$ be a $\Qp$-Banach algebra such that for every maximal ideal $\mathfrak{m}$ of $S$, $S/\mathfrak{m}$ is a finite extension of $\Qp$. Examples of such Banach algebras are $\Qp$-affinoid algebras. Let $\cal{X}$ be the set of maximal ideals of $S$. We write $\mathfrak{m}_x$ for the maximal ideal of $S$ corresponding to a point $x \in \cal{X}$.

Recall that $\pi = [\epsilon] - 1$, where $\epsilon \in \C_p^{\flat}$ is a fixed sequence of compatible, primitive $p^n$-th roots of unity. For each $r > 0$, $\Bt^{\dagger, r}$ is a locally convex vector space over $\Qp$ with the family of lattices $\pi^n\At^{\dagger, r}$ for integers $n \geq 0$. Therefore, we may define the completed tensor product $S \hat{\otimes}_{\Qp} \Bt^{\dagger, r}$ as the completion of the usual tensor product with respect to the given topology on $S$ and the above topology on $\Bt^{\dagger, r}$. We similarly define $S \hat{\otimes}_{\Qp} \B^{\dagger, r}_{\tau, K}$ and $S \hat{\otimes}_{\Qp} \Bt^{\dagger, r}_{L}$. Following \cite[Pg. 949]{KL10} we define $S\hat{\otimes}_{\Qp} \B^{\dagger}_{\tau, K} = \cup_{r > 0} S \hat{\otimes}_{\Qp}\B^{\dagger, r}_{\tau, K}$ and $S\hat{\otimes}_{\Qp} \Bt^{\dagger}_{L} = \cup_{r > 0} S \hat{\otimes}_{\Qp}\Bt^{\dagger, r}_{L}$.

A family of $p$-adic representations of $\G_K$ is an $S$-module $V$ free of finite rank $d$, endowed with a continuous linear action of $\G_K$.  For a point $x \in \cal{X}$, the specialization $S/\mathfrak{m}_x \otimes_{S} V$ is denoted by $V_x$.

A family of $\varphi$-modules over $S \hat{\otimes}_{\Qp} \B_{\tau, K}^{\dagger}$ is a locally free $S \hat{\otimes}_{\Qp} \B_{\tau, K}^{\dagger}$-module $D$ of finite rank with a Frobenius semilinear operator $\varphi_D$ such that the induced map $1 \otimes \varphi_D : S \hat{\otimes}_{\Qp} \B_{\tau, K}^{\dagger} \otimes_{\varphi, S \hat{\otimes}_{\Qp} \B_{\tau, K}^{\dagger}} D \to D$ is an isomorphism.

\begin{defi}\label{Overconvergent phi tau modules}
    A family of $(\varphi, \tau)$-modules over $(S \hat{\otimes}_{\Qp} \B^{\dagger}_{\tau, K}, S \hat{\otimes}_{\Qp} \Bt^{\dagger}_L)$ is a triple $(D, \varphi_D, G)$, where
    \begin{enumerate}
    \item $(D,\phi_D)$ is a $\phi$-module over $S \hat{\otimes}_{\Qp} \B^{\dagger}_{\tau, K}$,
    \item $G$ is \textit{the datum of} a continuous semilinear $G_\infty$-action on $M := (S \hat{\otimes}_{\Qp}\Bt_L^{\dagger}) \otimes_{S \hat{\otimes}_{\Qp}\B_{\tau,K}^{\dagger}}D$ such that $G_{\infty}$ commutes with $\phi_M:=\phi_{S \hat{\otimes}_{\Qp}\Bt_L^{\dagger}}\otimes \phi_D$, i.e., for all $g \in G_\infty$, we have $g\phi_M = \phi_Mg$,
    \item regarding $D$ as a sub-$S \hat{\otimes}_{\Qp} \B_{\tau,K}^{\dagger}$-module of $M$, $D \subset M^{H_{\tau,K}}$.
\end{enumerate}
\end{defi}

Since the ring $S \hat{\otimes}_{\Qp} \Bt^{\dagger, r}_{\tau, K}$ is not stable under $\varphi$, we define a family of $\varphi$-modules over $S \hat{\otimes}_{\Qp} \Bt^{\dagger, r}_{\tau, K}$ as is defined in the last paragraph on \cite[Page 1592]{Pot13}. A family of $(\varphi, \tau)$-modules over $(S \hat{\otimes}_{\Qp} \Bt^{\dagger, r}_{\tau, K}, S \hat{\otimes}_{\Qp} \Bt^{\dagger, r}_L)$ is defined similar to Definition~\ref{Overconvergent phi tau modules}.

The theorem proved by the second author and L\'{e}o Poyeton in \cite[Theorem 49]{KP23} is the following.
\begin{theo}
\label{theo ana families}
Let $V$ be a family of representations of $\G_K$ of rank $d$. Then there exists $r_0 > 0$ such that for any $r \geq r_0$, there exists a family of $(\varphi, \tau)$-modules $D_{\tau,K}^{\dagger,r}(V)$ such that
\begin{enumerate}
\item $D_{\tau,K}^{\dagger,r}(V)$ is an $S\hat{\otimes}_{\Qp}\B^{\dagger,r}_{\tau, K}$-module locally free of rank $d$,
\item the map $(S\hat{\otimes}_{\Qp}\Bt^{\dagger,r})\otimes_{S\hat{\otimes}_{\Qp}\B_{\tau,K}^{\dagger,r}}D_{\tau,K}^{\dagger,r}(V) \rightarrow (S\hat{\otimes}_{\Qp}\Bt^{\dagger,r})\otimes_S V$ is an isomorphism,
\item if $x \in \cal{X}$, the map $S/\mathfrak{m}_x\otimes_SD_{\tau,K}^{\dagger,r}(V) \rightarrow D_{\tau,K}^{\dagger,r}(V_x)$ is an isomorphism.
\end{enumerate}
\end{theo}

From this point on, all completed tensor products will be over $\Qp$. Using the existence result above, we make the following definition.

\begin{defi} \label{Definition of families of phi tau modules and their base change}
For a family $V$ of representations of $\G_K$, we define
    \begin{itemize}
        \item $D_{\tau, K}^{\dagger}(V) := S \hat{\otimes} \B_{\tau, K}^{\dagger} \otimes_{S \hat{\otimes}\B_{\tau, K}^{\dagger, r_0}} D_{\tau, K}^{\dagger, r_0}(V)$,

        \item $\Dt^{\dagger, r}_L(V) := S\hat{\otimes}\Bt^{\dagger, r}_L\otimes_{S \hat{\otimes}\B_{\tau, K}^{\dagger, r_0}} D_{\tau, K}^{\dagger, r_0}(V)$ for any $r \geq r_0$,

        \item $\Dt_{L}^{\dagger}(V) := S \hat{\otimes} \tilde{\B}_{L}^{\dagger} \otimes_{S \hat{\otimes} \B_{\tau, K}^{\dagger}} D_{\tau, K}^{\dagger}(V)$.
    \end{itemize}
\end{defi}

\begin{lemm}\label{Hinfty fixed points}
Let $r \geq r_0$. Then,
    \begin{enumerate}
        \item the canonical injection
        \[
            S \hat{\otimes} (\Bt^{\dagger, r})^{H_{\infty}} \hookrightarrow (S \hat{\otimes} \Bt^{\dagger, r})^{H_{\infty}}
        \]
        is an isomorphism.
        
        \item we have
        \[
            \Dt_{L}^{\dagger, r}(V) \simeq \left(S \hat{\otimes} \tilde{\B}^{\dagger, r} \otimes_{S \hat{\otimes} \B_{\tau, K}^{\dagger, r}} D_{\tau, K}^{\dagger, r}(V)\right)^{H_{\infty}}.
        \]
    \end{enumerate}
\end{lemm}
\begin{proof}
    Since $S$ is a $\Qp$-Banach space, there exists a set $I$ such that $S$ is topologically isomorphic to $\ell^{0}_{\infty}(I, \Qp)$, the $\Qp$-Banach space of functions $I \to \Qp$ which converge to $0$ in the filter of complements of finite subsets of $I$ (see, e.g., \cite[Proposition I.1.5]{Col10}).

    By \cite[Proposition I.1.8]{Col10}, we see that the canonical map $\ell^{0}_{\infty}(I, \Qp) \otimes_{\Qp} \Bt^{\dagger, r} \to \ell^{0}_{\infty}(I, \Bt^{\dagger, r})$ induces an isomorphism $\ell^{0}_{\infty}(I, \Qp) \hat{\otimes}_{\Qp} \Bt^{\dagger, r} \simeq \ell^{0}_{\infty}(I, \Bt^{\dagger, r})$. Since the action of $\G_K$ on $S$ is trivial, we see that $(S \hat{\otimes}_{\Qp} \Bt^{\dagger, r})^{H_{\infty}} \simeq \ell^{0}_{\infty}(I, (\Bt^{\dagger, r})^{H_{\infty}})$. Finally, $\ell^{0}_{\infty}(I, (\Bt^{\dagger, r})^{H_{\infty}}) \simeq \ell^{0}_{\infty}(I, \Qp) \hat{\otimes}_{\Qp}(\Bt^{\dagger, r})^{H_{\infty}} \simeq S \hat{\otimes}_{\Qp}(\Bt^{\dagger, r})^{H_{\infty}}$. This proves
    \[
        S \hat{\otimes} (\Bt^{\dagger, r})^{H_{\infty}} \simeq \left(S \hat{\otimes} \Bt^{\dagger, r}\right)^{H_{\infty}}.
    \]

    For brevity's sake, let $D = D_{\tau, K}^{\dagger, r}(V)$ and $\D_L = \Dt_L^{\dagger, r}(V)$. Since $D$ is projective by construction, there exists $D'$ such that the sequence
    \[
        0 \to D \to (S \hat{\otimes} \B_{\tau, K}^{\dagger, r})^n \to D' \to 0
    \]
    is exact. Using this sequence, we get the following diagram with exact rows (the tensor products with $D$ or $D'$ are over $S \hat{\otimes} \B_{\tau, K}^{\dagger, r}$)
   
    \[
        \begin{tikzcd}
            & (S \hat{\otimes} \Bt^{\dagger, r})^{H_{\infty}} \otimes D \ar[r] \ar[d] & \left((S \hat{\otimes} \Bt^{\dagger, r})^{H_{\infty}}\right)^{n} \ar[r] \ar[d] & (S \hat{\otimes} \Bt^{\dagger, r})^{H_{\infty}} \otimes_{} D' \ar[r] \ar[d] & 0 \\
            0 \ar[r] & \left((S \hat{\otimes} \Bt^{\dagger, r}) \otimes D\right)^{H_{\infty}} \ar[r] & \left((S \hat{\otimes} \Bt^{\dagger, r})^n\right)^{H_{\infty}} \ar[r] & \left((S \hat{\otimes} \Bt^{\dagger, r}) \otimes D'\right)^{H_{\infty}}.
        \end{tikzcd}
    \]
    Since the middle vertical map is an isomorphism and since the right vertical map is an injection, we see that the left vertical map is a surjection. The injectivity of the left vertical map is easy to see. Therefore using $(1)$, we see that
    \[
        \left((S \hat{\otimes} \Bt^{\dagger, r}) \otimes_{S \hat{\otimes} \B_{\tau, K}^{\dagger, r}} D\right)^{H_{\infty}} \simeq (S \hat{\otimes} \Bt^{\dagger, r}_L) \otimes_{S \hat{\otimes} \B_{\tau, K}^{\dagger, r}} D \simeq \D_L.\qedhere
    \]
\end{proof}

\section{Cohomology of $G_{\infty}$-modules and $H_{\infty}$-modules}
In this section, we first write down a complex that allows us to compute the cohomology of a $G_{\infty}$-module. Then we write down a complex that computes the cohomology of a family of $\G_K$-representations restricted to $H_{\infty}$. Finally,  we compute the cohomology of a family of $\G_K$-representations in terms of the cohomology of related modules for the groups $G_{\infty}$ and $H_{\infty}$ using the Hochschild-Serre spectral sequence.

\subsection{Cohomology of $G_{\infty}$-modules}

Recall that if $G$ is a procyclic group and $M$ is a continuous $G$-module, then $H^i(G, M)$ can be computed using the complex
\[
    0 \to M \xrightarrow{g - 1} M \to 0,
\]
where $g$ is a topological generator of $G$.

The group $G_{\infty}$ is not procyclic, but is topologically generated by two elements $\gamma$ and $\tau$. We write down a three-term complex that computes the cohomology of a $G_{\infty}$-module. This is a generalization of the two-term complex displayed above. 
This is needed in the proof of Theorem~\ref{thm:mainIntro}.

\begin{prop}\label{Explicit Ginfty cohomology}
        The cohomology of a $G_{\infty}$-module $M$ is isomorphic to that of the complex
        \[
            \mathcal{C}^{\bullet}(G_{\infty}, M) :=  M \xrightarrow{f_M = {\tiny \begin{pmatrix}\gamma - 1 \\ \tau - 1 \end{pmatrix}}} M \oplus M \xrightarrow{g_M = (\tau - 1, 1 - \delta^{-1}\gamma)} M .
        \]
        Moreover, the explicit isomorphisms are given as follows.
        \begin{enumerate}
            \item The isomorphism $H^0(G_{\infty}, M) \xrightarrow{\sim} \ker f_M$ is the identity map.

            \item The isomorphism $H^1(G_{\infty}, M) \xrightarrow{\sim} \dfrac{\ker g_M}{\im f_M}$ is induced by $c \mapsto (c(\gamma), c(\tau))$.

            \item The isomorphism $H^2(G_{\infty}, M) \xrightarrow{\sim} \coker~g_M$ is induced by $c \mapsto c(\gamma, \gamma^{-1}\tau\gamma) - c(\tau, \gamma)$. 
        \end{enumerate}
\end{prop}
    \begin{proof}
        Using equation \eqref{The monodromy identity}, the following computation
        \begin{eqnarray*}
            (1 - \delta^{-1}\gamma)(\tau - 1) + (\tau - 1)(\gamma - 1) & = & (\tau - 1) - \frac{\tau - 1}{\tau^{\chi(\gamma)} - 1} \gamma (\tau - 1) + (\tau - 1)(\gamma - 1) \\
            & = & -\frac{(\tau - 1)}{(\tau^{\chi(\gamma)} - 1)}(\tau^{\chi(\gamma)} - 1)\gamma + (\tau - 1)\gamma \\
            & = & 0
        \end{eqnarray*}
        shows that the displayed diagram is indeed a complex. Next, we show that $\mathcal{C}^{\bullet}(G_{\infty}, M)$ computes $H^i(G_{\infty}, M)$ for all $i \geq 0$. Indeed, we only have to check it for $0 \leq i \leq 2$ as the $p$-cohomological dimension of $G_{\infty}$ is less than or equal to $2$. The claim for $H^0(G_{\infty}, M)$ is clear, since $\gamma$ and $\tau$ generate $G_{\infty}$.

        \vspace{0.2cm}
        \noindent\underline{\textbf{Calculations for $H^1(G_{\infty}, M)$:} }
        For the claim involving $H^1(G_{\infty}, M)$, the Hochschild-Serre spectral sequence for the normal subgroup $\langle \tau \rangle \subseteq G_{\infty}$ says that
        \begin{eqnarray}\label{HS for H1 for a Ginfty module}
            0 \to H^1(\langle \gamma \rangle, M^{\tau = 1}) \xrightarrow{\mathrm{inf}} H^1(G_{\infty}, M) \xrightarrow{\mathrm{res}} \left(H^1(\langle \tau \rangle, M)\right)^{\gamma = 1} \to 0
        \end{eqnarray}
        is exact. It is easy to see that the ``evaluation-at-$\gamma$ map'' yields an isomorphism $H^1(\langle \gamma \rangle, M^{\tau = 1}) \xrightarrow{\sim} M^{\tau = 1}/(\gamma - 1)$. Similarly, we get an isomorphism
        \begin{eqnarray*}
            H^1(\langle \tau \rangle, M) & \xrightarrow{\sim} & M/(\tau - 1) \\
            c & \to & c(\tau).
        \end{eqnarray*}
        However, we must be careful with the $\gamma$-action on both sides. We claim that under the map above, we have $\big(H^1(\langle \tau \rangle, M)\big)^{\gamma = 1} \xrightarrow{\sim} \big(M/(\tau - 1)\big)^{\delta^{-1}\gamma = 1}$. Indeed, suppose $c \in H^1(\langle \tau \rangle, M)$ is fixed by $\gamma$. So
        \begin{eqnarray*}
            c(\tau) & = & (\gamma \cdot c)(\tau) \\
            & = & \gamma \left(c(\gamma^{-1} \tau \gamma)\right) \\
            & = & \gamma \left(c(\tau^{\chi(\gamma)^{-1}})\right) \\
            & = & \gamma\left(\frac{\tau^{\chi(\gamma)^{-1}} - 1}{\tau - 1}c(\tau)\right) \\
            & = & \delta^{-1}\gamma c(\tau).
        \end{eqnarray*}
        The third and the fifth equalities follow from \eqref{The monodromy identity}. The fourth equality follows from the cocycle condition.
        Therefore $\delta^{-1}\gamma(c(\tau)) = c(\tau)$. 

        Next, we claim that the following sequence is exact
        \begin{eqnarray*}
            0 \to M^{\tau = 1}/(\gamma - 1) \xrightarrow{\iota} & \frac{\ker g_M}{\im f_M} & \xrightarrow{\pi} \big(M/(\tau - 1)\big)^{\delta^{-1}\gamma = 1} \to 0 \\
            x \qquad \quad \mapsto & (x, 0) & \\
            & (x, y) & \mapsto \qquad \quad y.
        \end{eqnarray*}
        \begin{enumerate}
            \item Proof of exactness at $M^{\tau = 1}/(\gamma - 1)$: \\
            Given $x \in M^{\tau = 1}$ representing a class in $M^{\tau = 1}/(\gamma - 1)$, we see that $(\tau - 1)x = 0$. This means $(x, 0) \in \ker g_M$. Moreover, if $(x, 0) \in \im f_M$, then there exists $x' \in M$ such that $(\gamma - 1)x' = x$ and $(\tau - 1)x' = 0$. This means that $x$ belongs to $(\gamma - 1)M^{\tau = 1}$.

            \vspace{0.2cm}
            
            \item Proof of exactness at $\frac{\ker g_M}{\im f_M}$: 
            
            \vspace{0.2cm} 
            
            \noindent Given $(x, y) \in \ker g_M$, we see that $(\delta^{-1}\gamma - 1)y = (\tau - 1)x$. Therefore $\delta^{-1}\gamma y = y + (\tau - 1) x$. This shows that the class of $y$ in $M/(\tau - 1)$ belongs to $\big(M/(\tau - 1)\big)^{\delta^{-1}\gamma = 1}$. So the map $\pi$ is well-defined. Next, suppose $\pi(x, y) = 0$. Then there exists $y' \in M$ such that $y = (\tau - 1)y'$. Note that the difference $(x', 0) := (x, y) - ((\gamma - 1)y', (\tau - 1)y')$ belongs to $\ker g_M$ and represents the same class in $\frac{\ker g_M}{\im f_M}$ as $(x, y)$. In particular, $x'$ belongs to $M^{\tau = 1}$ and its class modulo $(\gamma - 1)M^{\tau = 1}$ maps under $\iota$ to the class of $(x, y)$.

            \vspace{0.2cm}
            
            \item Proof of exactness at $\big(M/(\tau - 1)\big)^{\delta^{-1}\gamma = 1}$: \\
            Given $y \in M$ representing a class in $\big(M/(\tau - 1)\big)^{\delta^{-1}\gamma = 1}$, we see that there exists an $x \in M$ such that $\delta^{-1}\gamma y = y + (\tau - 1)x$. Then $(x, y)$ is an element $\ker g_M$. The class of $(x, y)$ in $\frac{\ker g_M}{\im f_M}$ maps to the class of $y$ under $\pi$.
        \end{enumerate}

        Consider the following diagram where the vertical arrows are the natural evaluation maps
        \[
            \begin{tikzcd}
                0 \ar[r] & H^1(\langle \gamma \rangle, M^{\tau = 1}) \ar[r, "\mathrm{inf}"] \ar[d, "\mathrm{ev_{\gamma}}"] & H^1(G_{\infty}, M) \ar[r, "\mathrm{res}"] \ar[d, "{(\mathrm{ev_{\gamma}}, \mathrm{ev_{\tau}})}"] & \big(H^1(\langle \tau \rangle, M)\big)^{\gamma = 1} \ar[r] \ar[d, "\mathrm{ev_{\tau}}"] & 0 \\
                0 \ar[r] & M^{\tau = 1}/(\gamma - 1) \ar[r, "\iota"] & \frac{\ker g_M}{\im f_M} \ar[r, "\pi"] & \big(M/(\tau - 1) \big)^{\delta^{-1}\gamma = 1} \ar[r] & 0.
            \end{tikzcd}
        \]
        It is easy to see that the diagram commutes. Therefore $H^1(G_{\infty}, M) \simeq \frac{\ker g_M}{\im f_M}$ by snake lemma.

        \vspace{0.2cm}
        \noindent\underline{\textbf{Calculations for $H^2(G_{\infty}, M)$:} }
        For the claim involving $H^2(G_{\infty}, M)$, Hochschild-Serre spectral sequence gives
        \[
            H^2(G_{\infty}, M) \simeq H^1\big(\langle \gamma \rangle, H^1(\langle \tau \rangle, M)\big) \simeq \coker~g_M.
        \]
        The first isomorphism can be described using \cite[Section 10.3]{DHW12} as follows. A cocycle $c : G_{\infty}\times G_{\infty} \to M$ representing an element in $H^2(G_{\infty}, M)$ is mapped to  the cocycle in $H^1\big(\langle \gamma \rangle, H^1(\langle \tau \rangle, M)\big)$ by sending
        \[
            \gamma \mapsto \left(\tau \mapsto c(\gamma, \gamma^{-1}\tau \gamma) - c(\tau, \gamma)\right).
        \]
        Since the second isomorphism is the composition of evaluations at $\tau$ and $\gamma$, we see that the ismorphism $H^2(G_{\infty}, M) \xrightarrow{\sim} \coker~g_M$ is given by
        \[
            [c] \mapsto c(\gamma, \gamma^{-1}\tau\gamma) - c(\tau, \gamma). \qedhere
        \]   
    \end{proof}

\subsection{Cohomology of $H_{\infty}$-modules}\label{Cohomology of Hinfty modules}

Next, we are interested in the cohomology of continuous $H_{\infty}$-modules. We need the following preparatory lemmas.

\begin{lemm}\label{cohomology of BL}
    For $r > 0$, we have
    \[
        H^{i}(H_{\infty}, \Bt^{\dagger, r}) = \begin{cases}
            \Bt^{\dagger, r}_{L} & \text{ if } i = 0, \\
            0 & \text{ if } i > 0.
        \end{cases}
    \]
\end{lemm}
\begin{proof}
    The $i = 0$ part of this lemma follows from the definition of $\Bt_{L}^{\dagger, r}$. So now assume that $i > 0$.
    Since $H_{\infty}$ is compact, we see that
    \begin{eqnarray}\label{Rational and integral Hinfinity-cohomology groups}
        H^i(H_{\infty}, \Bt^{\dagger, r}) \simeq H^i(H_{\infty}, \At^{\dagger, r}) \otimes_{\Zp} \Qp.
    \end{eqnarray}
    The $i > 1$ part follows from \cite[Proposition 2.4.10]{Pon} by noting that $H_{\infty}$ has $p$-cohomological dimension $\leq 1$ (since it is the absolute Galois group of a characteristic $p$ field by Fontaine-Winterberger). 

    Next, we note that $H^1(H_{\infty}, \At^{\dagger, r}) \simeq \varprojlim_n H^1(H_{\infty}, \At^{\dagger, r}/p^n)$ by \cite[Theorem 2.7.5]{Neu} (see also \cite[Proposition 2.1.1]{Pon}). The derived inverse limit there is $0$ because the action of $H_{\infty}$ on $\At^{\dagger, r}$ is only through its action on the teichm\"uller lifts. By induction on $n$, if we show $H^1(H_{\infty}, \C^{\flat}_p) = 0$, then we may conclude $H^1(H_{\infty}, \At^{\dagger, r}/p^n) = 0$. This, in turn, follows from Hilbert-$90$ applied to the extension $\C^{\flat}_p/\C_p^{\flat, H_{\infty}}$. \qedhere

\end{proof}

\begin{lemm}\label{Cocycles commute with tensor S}
    For a real number $r > 0$ and an integer $i \geq 0$, we have the following property concerning $i$-cochains
    \[
        S \hat{\otimes} \C^{i}(H_{\infty}, \Bt^{\dagger, r}) \simeq \C^{i}(H_{\infty}, S \hat{\otimes} \Bt^{\dagger, r}).
    \]
\end{lemm}
\begin{proof}
    Let $S^0$ denote the subset of elements in $S$ with norm less than or equal to $1$. We know that
    \begin{eqnarray*}
        S \hat{\otimes}_{\Qp} \Bt^{\dagger, r} \simeq \varprojlim_{n \geq 0} (S \otimes_{\Qp} \Bt^{\dagger, r})/p^n(S^0 \otimes_{\Zp} \At^{\dagger, r}) \simeq \varprojlim_{n \geq 0} (S^0 \otimes_{\Zp} \Bt^{\dagger, r})/p^n(S^0 \otimes_{\Zp} \At^{\dagger, r}). 
    \end{eqnarray*}
    Therefore, $\mathrm{Map}_{\cont}(H_{\infty}^i, S \hat{\otimes}_{\Qp} \Bt^{\dagger, r}) \simeq \varprojlim_{n \geq 0}\mathrm{Map}_{\cont}(H_{\infty}^i, S^0/p^nS^0 \otimes_{\Zp} \Bt^{\dagger, r}/p^n\At^{\dagger, r})$. Using the compactness of $H_{\infty}^i$ and the discreteness of the coefficient modules, we get
    \begin{eqnarray}\label{Continuous maps tensor product}
        \mathrm{Map}_{\cont}(H_{\infty}^i, S \hat{\otimes}_{\Qp} \Bt^{\dagger, r}) \simeq \varprojlim_{n \geq 0} S^0/p^nS^0 \otimes_{\Zp} \mathrm{Map}_{\cont}(H_{\infty}^i, \Bt^{\dagger, r}/p^n\At^{\dagger, r}).
    \end{eqnarray}
    
    Now, consider the following sequence
    \begin{eqnarray}\label{Continuous maps commute with quotients}
        0 \to \mathrm{Map}_{\cont}(H_{\infty}^i, p^n\At^{\dagger, r}) \to \mathrm{Map}_{\cont}(H_{\infty}^i, \Bt^{\dagger, r}) \to \mathrm{Map}_{\cont}(H_{\infty}^i, \Bt^{\dagger, r}/p^n\At^{\dagger, r}) \to 0.
    \end{eqnarray}
    The exactness at the first two modules is easy to check.
    Here is the argument for the exactness at the last module. Any continuous function $f$ from the compact set $H_{\infty}^i$ to the discrete set $\Bt^{\dagger, r}/p^n\At^{\dagger, r}$ factors through a finite quotient $(H_{\infty}/H)^i$ of $H_{\infty}^i$. We then produce a continuous function $f'$ from $(H_{\infty}/H)^i$ to $\Bt^{\dagger, r}$ by lifting each $f(x)$ arbitrarily for each $x \in (H_{\infty}/H)^i$. Then pre-composing $f'$ with the canonical surjection $H_{\infty}^i \to (H_{\infty}/H)^i$, we get a lift of $f$ valued in $\Bt^{\dagger, r}$.

    Therefore, combining \eqref{Continuous maps tensor product} and \eqref{Continuous maps commute with quotients}, we see that
    \[
        \mathrm{Map}_{\cont}(H_{\infty}^i, S \hat{\otimes}_\Qp \Bt^{\dagger, r}) \simeq \varprojlim_{n \geq 0} S^0/p^nS^0 \otimes_{\Zp} \left(\mathrm{Map}_{\cont}(H_{\infty}^i, \Bt^{\dagger, r})/\mathrm{Map}_{\cont}(H_{\infty}^i, p^n\At^{\dagger, r})\right).
    \]
    The inverse limit on the right is $S \hat{\otimes}_{\Qp}\left(\mathrm{Map}_{\cont}(H_{\infty}^i, \Bt^{\dagger, r})\right)$ is by definition. 
    
    One can check that under the isomorphism
    \[
        S \hat{\otimes}_{\Qp} \mathrm{Map}_{\cont}(H_{\infty}^i, \Bt^{\dagger, r}) \xrightarrow{\sim} \mathrm{Map}_{\cont}(H_{\infty}^i, S \hat{\otimes}_\Qp \Bt^{\dagger, r}),
    \]
    a function $f \in \mathrm{Map}_{\cont}(H_{\infty}^i, \Bt^{\dagger, r})$ is sent to its composition with $\Bt^{\dagger, r} \to \Bt^{\dagger, r} \hat{\otimes}_{\Qp} S$. This shows that the isomorphism commutes with differentials, proving the lemma.
\end{proof}

Now we prove the following proposition, which allows us to compute the $H_{\infty}$-cohomology of a family of representations of $\G_K$.
\begin{prop}\label{Hinfty cohomology}
    Let $V$ be a family of representations of $\G_K$ over $S$. Then, for $r \geq r_0$ we have
    \begin{enumerate}
        \item $\mathrm{H}^{0}(H_{\infty}, V) \simeq \ker\left(\Dt^{\dagger, r}_{L}(V) \xrightarrow{\varphi - 1} \Dt^{\dagger, pr}_{L}(V)\right)$,
        \item $\mathrm{H}^{1}(H_{\infty}, V) \simeq \coker\left(\Dt^{\dagger, r}_{L}(V) \xrightarrow{\varphi - 1} \Dt^{\dagger, pr}_{L}(V)\right)$,
        \item $H^i(H_{\infty}, V) = 0$ if $i \geq 2$.
    \end{enumerate}
    Therefore, the complex
    \[
        0 \to \Dt^{\dagger, r}_L(V) \xrightarrow{\varphi-1} \Dt^{\dagger, pr}_L(V) \to 0
    \]
    computes the cohomology of the $H_{\infty}$-module $V$.
\end{prop}
\begin{proof}
    Consider the following well-known short exact sequence
    \[
        0 \to \Qp \to \Bt^{\dagger, r} \xrightarrow{\varphi - 1} \Bt^{\dagger, pr} \to 0.
    \]
    Since $S$ is a $\Qp$-Banach algebra, taking completed tensor products with $S$ is the same as applying the functor $\ell^{0}_{\infty}(I, \_)$ for some fixed indexing set $I$ depending on $S$. So we immediately see that taking completed tensor products with $S$ is an exact functor. Therefore the following sequence is exact
    \[
        0 \to S \to S \hat{\otimes}_{\Qp} \Bt^{\dagger, r} \xrightarrow{\varphi - 1} S \hat{\otimes}_{\Qp} \Bt^{\dagger, pr} \to 0.
    \]
    Recall that a family $V$ of representations of $\G_K$ is a free $S$-module. It is therefore flat over $S$. So we may tensor the exact sequence above with $V$ to get
    \[
        0 \to V \to (S \hat{\otimes}_{\Qp} \Bt^{\dagger, r}) \otimes_S V \to (S \hat{\otimes}_{\Qp}\Bt^{\dagger, pr}) \otimes_S V \to 0.
    \]
    Rewriting the second and third terms using Theorem~\ref{theo ana families} (2), the sequence above becomes
    \[
        0 \to V \to (S \hat{\otimes}\Bt^{\dagger, r}) \otimes_{S \hat{\otimes}\B^{\dagger, r}_{\tau, K}} \D^{\dagger, r}_{\tau, K}(V) \to (S \hat{\otimes}\Bt^{\dagger, pr}) \otimes_{S \hat{\otimes}\B^{\dagger, pr}_{\tau, K}} \D^{\dagger, pr}_{\tau, K}(V) \to 0.
    \]
    We claim that the lemma follows once it is shown that $\mathrm{H}^{i}(H_\infty, (S \hat{\otimes}\Bt^{\dagger, (r, pr)}) \otimes_{S \hat{\otimes}\B^{\dagger, (r, pr)}_{\tau, K}} \D^{\dagger, (r, pr)}_{\tau, K}(V)) = 0$ for $i \geq 1$. Indeed, taking $H_{\infty}$-fixed points, we get a long exact sequence of cohomology groups because of the existence of an $S$-linear continuous section of the surjection $\varphi - 1 : (S \hat{\otimes}\Bt^{\dagger, r}) \otimes_{S \hat{\otimes}\B^{\dagger, r}_{\tau, K}} \D^{\dagger, r}_{\tau, K}(V) \to (S \hat{\otimes}\Bt^{\dagger, pr}) \otimes_{S \hat{\otimes}\B^{\dagger, pr}_{\tau, K}} \D^{\dagger, pr}_{\tau, K}(V)$ obtained using a continuous $\Qp$-linear section to $\varphi - 1 : \Bt^{\dagger, r} \to \Bt^{\dagger, pr}$.

    First, we use Lemma~\ref{cohomology of BL} to get the following exact sequence 
   
    \[
        0 \to \Bt^{\dagger, (r, pr)}_{L} \to \C^0_{\cont}(H_\infty, \Bt^{\dagger, (r, pr)}) \to \C^1_{\cont}(H_\infty, \Bt^{\dagger, (r, pr)}) \to \cdots.
    \]
    Next, we apply $S \hat{\otimes} \_$, use Lemma~\ref{Cocycles commute with tensor S}, and apply $\_ \otimes \D^{\dagger, (r, pr)}_{\tau, K}(V)$ to get $$\mathrm{H}^{i}(H_\infty, (S \hat{\otimes}\Bt^{\dagger, (r, pr)}) \otimes_{S \hat{\otimes}\B^{\dagger, (r, pr)}_{\tau, K}} \D^{\dagger, (r, pr)}_{\tau, K}(V)) = 0$$ for $i \geq 1$.
\end{proof}

It is worth making the isomorphisms in Proposition~\ref{Hinfty cohomology} explicit for the proof of Theorem~\ref{thm:mainIntro}. So we do it here. We fix, once and for all, a continuous $S$-linear section $s^{(pr)} : (S \hat{\otimes}_{\Qp} \Bt^{\dagger, pr}) \otimes_S V \to (S \hat{\otimes}_{\Qp} \Bt^{\dagger, r}) \otimes_S V$ to $\varphi - 1$.

The first isomorphism in the lemma above is induced by the canonical inclusion $V^{H_{\infty}} \hookrightarrow \left((S \hat{\otimes}\Bt^{\dagger, r}) \otimes_S V\right)^{H_{\infty}}$ followed by the isomorphisms $\left((S \hat{\otimes}\Bt^{\dagger, r}) \otimes_S V\right)^{H_{\infty}} \simeq \left(S \hat{\otimes} \tilde{\B}^{\dagger, r} \otimes_{S \hat{\otimes} \B_{\tau, K}^{\dagger, r}} \D_{\tau, K}^{\dagger, r}(V)\right)^{H_{\infty}} \simeq \Dt_{L}^{\dagger, r}(V)$ obtained using Theorem~\ref{theo ana families} (2) and Lemma~\ref{Hinfty fixed points} (2).

We let $\eta^{(pr)} : \mathrm{H}^{1}(H_{\infty}, V) \xrightarrow{\sim} \coker\left(\Dt^{\dagger, r}_{L}(V) \xrightarrow{\varphi - 1} \Dt^{\dagger, pr}_{L}(V)\right)$ be the second isomorphism.
It is described using the connecting homomorphism in the following way. Let $x \in \Dt^{\dagger, pr}_L(V)$ represent an element in $\frac{\Dt^{\dagger, pr}_L(V)}{(\varphi - 1)\Dt^{\dagger, r}_L(V)}$. Then, under the second isomorphism in the proposition above, the class of $x$ corresponds to the $1$-cocycle given by $\sigma \mapsto (\sigma - 1)s^{(pr)}x$ for $\sigma \in H_{\infty}$. In particular, every class $\alpha \in H^1(H_{\infty}, V)$ is represented by a $1$-cocycle $c$ defined by $c(\sigma) = (\sigma - 1)s^{(pr)}x$ for an $x \in \Dt^{\dagger, pr}_L(V)$ representing the class $\eta^{(pr)}(\alpha)$.

\subsection{The inflation-restriction sequence for $H_{\infty}$ and $G_{\infty}$}\label{Explicit maps in inflation-restriction}

The inflation-restriction exact sequence with explicit maps is a key tool in the proof of Theorem~\ref{thm:mainIntro}. To this end, we adopt the results of Dekimpe, Hartl, and Wauters \cite{DHW12} to our setting.

As usual, let $V$ be a family of representations of $\G_K$. Then, the exact sequence on Page $71$ of loc. cit. applied to the normal subgroup $H_{\infty}$ of $\G_K$ becomes
\begin{equation}\label{DHW exact sequence}
    \begin{aligned}
        0 & \to H^1(G_{\infty}, V^{H_{\infty}}) \xrightarrow{\mathrm{\inf}} H^1(\G_K, V) \xrightarrow{\mathrm{res}} H^1(H_{\infty}, V)^{G_{\infty}} \xrightarrow{\mathrm{tr}} H^2(G_{\infty}, V^{H_{\infty}}) \\
        & \xrightarrow{\mathrm{inf}} H^2(\G_K, V) \xrightarrow{\rho}H^1(G_{\infty}, H^1(H_{\infty}, V)) \to 0.
    \end{aligned}
\end{equation}
Since $G_{\infty}$ is an extension of two groups of $p$-cohomological dimension $1$, it does not support cohomologies in degree $3$ and above. Therefore we have omitted the $H^3(G_{\infty}, V^{H_{\infty}})$ term. Also, since $H^2(H_{\infty}, V) = 0$ by Proposition~\ref{Hinfty cohomology}, we see that the group $H^2(\G_K, V)_1$ of loc. cit. is just $H^2(\G_K, V)$.

Let us describe the maps appearing in \eqref{DHW exact sequence}.  
Here, $\mathrm{inf}$ and $\mathrm{res}$ are the usual inflation and restriction maps. In the following, we describe the transgression map $\mathrm{tr}$ from \cite[Proposition 1.6.6]{Neu} and the map $\rho$ from \cite[Section 10.3]{DHW12}. 

Let $x : H_{\infty} \to V$ be a $1$-cocycle representing a class in $H^1(H_{\infty}, V)^{G_{\infty}}$. Then there exists a \underline{$1$-cochain} $y : \G_K \to V$ such that 
\begin{itemize}
    \item $y\vert_{H_{\infty}} = x$,
    \item $\partial y : \G_K\times\G_K \to V$ defined by $\partial y(\sigma_1, \sigma_2) = y(\sigma_1) - y(\sigma_1\sigma_2) + \sigma_1 y(\sigma_2)$ maps $\G_K \times \G_K$ to $V^{H_{\infty}}$,
    \item the value $\partial y(\sigma_1, \sigma_2)$ depends only on the classes of $\sigma_1$ and $\sigma_2$ modulo $H_{\infty}$.
\end{itemize}
The map $\mathrm{tr}$ sends the class of $x$ to the class of $\partial y$ in $H^2(G_{\infty}, V^{H_{\infty}})$.

Next we describe the map $\rho$. Let $c : \G_K \times \G_K \to V$ be a $2$-cocycle representing an element in $H^2(\G_K, V)$. Then, for every $g \in \G_K$, the map $d_g : H_{\infty} \to V$ defined by $d_g(h) = c(g, g^{-1}hg) - c(h, g)$ is a $1$-cocycle in $H^1(H_{\infty}, V)$. Furthermore, the class of $d_g$ in $H^1(H_{\infty}, V)$ depends only on the class of $g$ modulo $H_{\infty}$. Thus, to the $2$-cocycle $c$, we may associate the $1$-cocycle $d_{\_} : G_{\infty} \to H^1(H_{\infty}, V)$ and define $\rho([c])$ to be the class of $d_{\_}$.

Even though these maps are defined for the classical cohomology groups, one can check that the description of the maps given above goes through for continuous cohomolgy groups mutatis mutandis, making the sequence \eqref{DHW exact sequence} exact.

\section{A four-term complex for $\G_K$-cohomology}
    In this section, we construct a four-term complex that computes the cohomology of families of representations of $\G_K$. 

    Let $V$ be a family of representations of $\G_K$ and for $r \geq r_0$, where $r_0$ is defined in Theorem~\ref{theo ana families}, recall the associated family of $(\varphi, \tau)$-modules $\D^{\dagger, r}_{\tau, K}(V)$ and its base change $\Dt^{\dagger, r}_L(V)$ to $S\hat{\otimes}\Bt^{\dagger, r}_L$ as in Definition~\ref{Definition of families of phi tau modules and their base change}. To simplify the notation, we let $\D^{(r, pr)} = \Dt^{\dagger, (r, pr)}_L(V)$.
    
    Consider the complex
    {\small
    \[
        0 \to \D^{r} \xrightarrow{d_0 = \begin{pmatrix}\varphi - 1 \\ \gamma - 1 \\ \tau - 1\end{pmatrix}} \D^{pr} \oplus \D^{r} \oplus \D^{r} \xrightarrow{d_1 = \begin{pmatrix}\gamma - 1 & 1 - \varphi & 0 \\ \tau - 1 & 0 & 1 - \varphi \\ 0 & \tau - 1 & 1 - \delta^{-1}\gamma\end{pmatrix}} \D^{pr} \oplus \D^{pr} \oplus \D^{r} \xrightarrow{d_2 = (\tau - 1, 1 - \delta^{-1}\gamma, \varphi - 1)} \D^{pr} \to 0.
    \]
    }

    Recall that for a $G_{\infty}$-module $M$, we have defined the maps $f_M, g_M$ in Proposition~\ref{Explicit Ginfty cohomology}. In the proof of the following proposition, we will repeatedly use the fact that $g_M \circ f_M = 0$.
    We first write a six-term exact sequence that is the $(\varphi, \tau)$-module analogue of sequence \eqref{DHW exact sequence}.
    \begin{prop}\label{Substitute for 6 term HS}
        With the notation as above, the following sequence
        \begin{eqnarray*}
            \begin{tikzcd}[column sep = 0.2cm]
            0 \ar[r] & \frac{\ker g_{(\D^{r})^{\varphi = 1}}}{\im f_{(\D^{r})^{\varphi = 1}}} \ar[r, "\delta_1"] & \frac{\ker d_1}{\im d_0} \ar[r, "\delta_2"] & \left(\frac{\D^{pr}}{(\varphi - 1)\D^{r}}\right)^{\tau = 1, \gamma = 1} \ar[r, "\delta_3"] & \frac{(\D^{r})^{\varphi = 1}}{(\tau - 1, 1 - \delta^{-1}\gamma)} \ar[r, "\delta_4"] & \frac{\ker d_2}{\im d_1} \ar[r, "\delta_5"] & \frac{\ker g_{(\D^{pr}/(\varphi - 1)\D^{r})}}{\im f_{(\D^{pr}/(\varphi - 1)\D^{r})}} \ar[r] & 0 \\
            & (y, z) \ar[r, mapsto] & (0, y, z) & x \ar[r, mapsto] & (\tau - 1)y_x + (1 - \delta^{-1}\gamma)z_x & (x, y, z) \ar[r, mapsto] & (x, y) & \\
            & & (x, y, z) \ar[r, mapsto] & x & z \ar[r, mapsto] & (0, 0, z)
            \end{tikzcd}
        \end{eqnarray*}
        is exact, where $y_x, z_x \in \D^{r}$ satisfy $(\gamma - 1)x = (\varphi - 1)y_x$ and $(\tau - 1)x = (\varphi - 1)z_x$.
    \end{prop}
    \begin{proof}
        It is easy to see that all maps except $\delta_3$ are well-defined. To see that $\delta_3$ is well-defined, let $x \in \D^{pr}$ represent an element in $\left(\frac{\D^{pr}}{(\varphi - 1)\D^{r}}\right)^{\tau = 1, \gamma = 1}$ and let $y_x'$ and $z_x'$ be another such choice. Then, $(\varphi - 1)(y_x - y_x') = (\gamma - 1)x - (\gamma - 1)x = 0$ and $(\varphi - 1)(z_x - z_x') = (\tau - 1)x - (\tau - 1)x = 0$. So, 
        \[
            [(\tau - 1)y_x + (1 - \delta^{-1}\gamma)z_x] - [(\tau - 1)y_x' + (1 - \delta^{-1}\gamma)z_x'] \in (\tau - 1)(\D^{r})^{\varphi = 1} + (1 - \delta^{-1}\gamma)(\D^{r})^{\varphi = 1}.
        \]
        This shows that $\delta_3$ is well-defined.
        
        One can check that the composition of any two successive maps is $0$. We now check exactness at each module:
        \begin{enumerate}
            \item Suppose $\delta_1(y, z) = 0$. This means that there exists $x \in \D^{r}$ such that $(\varphi - 1)x = 0$, $(\gamma - 1)x = y$ and $(\tau - 1)x = z$. So $x \in (\D^{r})^{\varphi = 1}$ and hence $(y, z) \in \im f_{(\D^{r})^{\varphi = 1}}$.
            \vspace{0.2cm}
            \item Suppose $\delta_2(x, y, z) = 0$. This means that there exists $x' \in \D^{r}$ such that $x = (\varphi - 1)x'$. Define $y' = y - (\gamma - 1)x'$ and $z' = z - (\tau - 1)x'$. We note that $(x, y, z) = (0, y', z')$ in the quotient $\frac{\ker d_1}{\im d_0}$. We check that $(y', z') \in \ker g_{(\D^{r})^{\varphi = 1}}$. Indeed, 
            \[
                (\varphi - 1)y' = (\varphi - 1)y - (\gamma - 1)(\varphi - 1)x' = (\varphi - 1)y - (\gamma - 1)x = 0
            \]
            since $(x, y, z) \in \ker d_1$. This shows that $y' \in (\D^{r})^{\varphi = 1}$. A similar check shows that $z' \in (\D^{r})^{\varphi = 1}$. Finally, $(\tau - 1)y' + (1 - \delta^{-1}\gamma)z' = (\tau - 1)y + (1 - \delta^{-1}\gamma)z - (\tau - 1)(\gamma - 1)x' - (1 - \delta^{-1}\gamma)(\tau - 1)x' = 0$. Therefore $(y', z') \in \ker g_{(\D^{r})^{\varphi = 1}}$ maps to the class of $(x, y, z)$ in $\frac{\ker d_1}{\im d_0}$.
            \vspace{0.2cm}
            \item Suppose $x \in \D^{pr}$ represents a class in $\left(\frac{\D^{pr}}{(\varphi - 1)\D^{r}}\right)^{\tau = 1, \gamma = 1}$ mapping to $0$ under $\delta_3$. Pick $y_x, z_x \in \D^{r}$ such that $(\gamma - 1)x = (\varphi - 1)y_x \text{ and }  (\tau - 1)x = (\varphi - 1)z_x.$
            Now assume that there exist $y_x', z_x' \in (\D^{r})^{\varphi = 1}$ such that $(\tau - 1)y_x + (1 - \delta^{-1}\gamma)z_x = (\tau - 1)y_x' + (1 - \delta^{-1}\gamma)z_x'$. We claim that $(x, y_x - y_x', z_x - z_x') \in \ker d_1$. Indeed, 
            \[
                (\gamma - 1)x + (1 - \varphi)(y_x - y_x') = (\gamma - 1)x + (1 - \varphi)y_x - (1 - \varphi)y_x' = 0.
            \]
            Similarly, $(\tau - 1)x + (1 - \varphi)(z_x - z_x') = (\tau - 1)x + (1 - \varphi)z_x - (1 - \varphi)z_x' = 0$. Finally, we also have $(\tau - 1)(y_x - y_x') + (1 - \delta^{-1}\gamma)(z_x - z_x') = 0$ using the definition of $y_x'$ and $z_x'$.
            \vspace{0.2cm}
            \item Suppose $z  \in (\D^{r})^{\varphi = 1}$ represents a class in $\frac{(\D^{r})^{\varphi = 1}}{(\tau - 1, 1 - \delta^{-1}\gamma)}$ mapping to $0$ under $\delta_4$. This means that there exists $(x', y', z') \in \D^{pr} \oplus \D^{r} \oplus \D^{r}$ such that $d_1(x', y', z') = (0, 0, z)$. Writing this explicitly, we see that $x' \in \left(\D^{pr}/(\varphi - 1)\D^{r}\right)^{\tau = 1, \gamma = 1}$ and $(\tau - 1)y' + (1 - \delta^{-1}\gamma)z' = z$. So, $z = \delta_3(x')$.
            \vspace{0.2cm}
            \item Suppose $\delta_5(x, y, z) = 0$. This means that there exist $x', y' \in \D^{r}$ and $z' \in \D^{pr}$ such that $(\gamma - 1)z' = x + (\varphi  - 1)x'$ and $(\tau - 1)z' = y + (\varphi - 1)y'$. Therefore, 
            \[
                (x, y, z) - d_1(z', x', y') = (0, 0, z - (\tau - 1)x' - (1 - \delta^{-1}\gamma)y').
            \]
            To show that this element belongs to the image of $\delta_4$, we only need to show that $z - (\tau - 1)x' - (1 - \delta^{-1}\gamma)y' \in (\D^{r})^{\varphi = 1}$. Indeed, 
            \begin{eqnarray*}
                (\varphi - 1)(z - (\tau - 1)x' - (1 - \delta^{-1}\gamma)y') & = & (\varphi - 1)z - (\tau - 1)\left((\gamma - 1)z' - x\right) \\ 
                && \quad - (1 - \delta^{-1}\gamma)((\tau - 1)z' - y) \\
                & = & d_2(x, y, z) = 0.
            \end{eqnarray*}
            
            \item Suppose $x, y \in \D^{pr}$ are such that $(x, y) \in \ker g_{(\D^{pr}/(\varphi - 1)\D^{r})}$. Then, there exists $z \in \D^{r}$ such that $(\tau - 1)x + (1 - \delta^{-1}\gamma)y = (\varphi - 1)(-z)$. This shows that $(x, y, z) \in \ker d_2$ and its class in $\frac{\ker d_2}{\im d_1}$ maps to the class of $(x, y)$ in $\frac{\ker g_{(\D^{pr}/(\varphi - 1)\D^{r})}}{\im f_{(\D^{pr}/(\varphi - 1)\D^{r})}}$ under $\delta_5$. \qedhere
        \end{enumerate}
    \end{proof}

\section{Explicit maps from group cohomology to cohomology of $(\varphi, \tau)$-modules}

    To show that $H^1(\G_K, V)$ is isomorphic to $\frac{\ker d_1}{\im d_0}$ and that $H^2(\G_K, V)$ is isomorphic to $\frac{\ker d_2}{\im d_1}$, we produce maps between complex \eqref{DHW exact sequence} and the complex written in Proposition~\ref{Substitute for 6 term HS}, and use the five-lemma.

\subsection{Computations for $H^1(\G_K, V)$}
    The following proposition gives us an explicit map from $H^1(\G_K, V)$ to $\frac{\ker d_1}{\im d_0}$.

    \begin{prop}\label{Definition of h1}
        Let $\alpha \in H^1(\G_K, V)$. Then, there exists a $1$-cocycle $c : \G_K \to V$ representing $\alpha$ such that $c(h) = (h - 1)s^{(pr)}x$ for some $x \in \Dt^{\dagger, pr}_L(V)$ congruent to $\eta^{(pr)}(\alpha)$ mod $(\varphi - 1)\Dt^{\dagger, r}_L(V)$. Moreover, for any $g \in \G_K$, the element $c(g) - (g - 1)s^{(pr)}x \in (S \hat{\otimes}_{\Qp} \Bt^{\dagger, r}) \otimes_S V$ belongs to $\Dt^{\dagger, r}_L(V)$, i.e., it is $H_{\infty}$-fixed. Furthermore, the map
        \begin{eqnarray*}
                h_1 : H^1(\G_K, V) & \to & \frac{\ker d_1}{\im d_0} \\
                \alpha & \mapsto & (-x, c(\tgamma) - (\tgamma - 1)s^{(pr)}x, c(\ttau) - (\ttau - 1)s^{(pr)}x) \mod \im d_0
            \end{eqnarray*}
            is independent of the choices of $c$ and $x$, and is an $S$-linear homomorphism.
    \end{prop}

    \begin{proof}
        Let $\alpha \in H^1(\G_K, V)$ be represented by a $1$-cocycle $c'$. Let $x \in \Dt^{\dagger, pr}_L(V)$ represent the class $\eta^{(pr)}([c'\vert_{H_{\infty}}])$ modulo $(\varphi - 1)\Dt^{\dagger, r}_L(V)$. Then, by the discussion at the end of Section~\ref{Cohomology of Hinfty modules}, there exists $v \in V$ such that $c'(h) = (h - 1)s^{(pr)}x + (h - 1)v$. Then, the $1$-cocycle $c : \G_K \to V$ defined by $c(g) = c'(g) - (g - 1)v$ also represents $\alpha$ and $c(h) = (h - 1)s^{(pr)}x$ for all $h \in H_{\infty}$.
    
        For any $h \in H_{\infty}$ and $g \in \G_K$, the cocycle condition implies that
        \[
            h\left(c(g) - (g - 1)s^{(pr)}x\right) = c(hg) - c(h) - hg s^{(pr)}x + hs^{(pr)}x.
        \]
        Using $c(h) = (h - 1)s^{(pr)}x$, we write
        \[
            c(hg) - c(h) - hgs^{(pr)}x + hs^{(pr)}x = c(hg) - hgs^{(pr)}x + s^{(pr)}x.
        \]
        Since $H_{\infty}$ is a normal subgroup of $\G_K$, there exists $h' \in H_{\infty}$ such that $hg = gh'$. So we substitute this in the RHS of the equation above and use the cocycle condition to write
        \[  
            c(gh') - gh's^{(pr)}x + s^{(pr)}x = c(g) + gc(h') - gh' s^{(pr)}x + s^{(pr)}x.
        \]
        Now we substitute $c(h') = (h' - 1)s^{(pr)}x$ in the last expression to get
        \[
            c(g) + gc(h') - gh' s^{(pr)}x + s^{(pr)}x = c(g) - (g - 1)s^{(pr)}x.
        \]
        Therefore we have shown that $H_{\infty}$ fixes $c(g) - (g - 1)s^{(pr)}x$. This shows that the tuple $$(-x, c(\tgamma) - (\tgamma - 1)s^{(pr)}x, c(\ttau) - (\ttau - 1)s^{(pr)}x)$$ belongs to $\Dt^{\dagger, pr}_L(V) \oplus \Dt^{\dagger, r}_L(V) \oplus \Dt^{\dagger, r}_L(V)$. 

        Now we show that $h_1$ is independent of the choices of $v, x$, and $c'$.
        \begin{itemize}
            \item Fix $c'$ and $x$. Let $v$ and $w$ be elements in $V$ such that $c'(h) = (h - 1)s^{(pr)}x + (h - 1)v$ and $c'(h) = (h - 1)s^{(pr)}x + (h - 1)w$. In particular, $w - v \in V^{H_{\infty}}\subset \Dt^{\dagger, r}_L(V)$. Let $c_v, c_w : \G_K \to V$ be defined by $c_v(g) = c'(g) - (g - 1)v$ and $c_w(g) = c'(g) - (g - 1)w$. Then, the difference
            \begin{eqnarray*}
                 && \left(-x, c_v(\tgamma) - (\tgamma - 1)s^{(pr)}x, c_v(\ttau) - (\ttau - 1)s^{(pr)}x\right) - \left(-x, c_w(\tgamma) - (\tgamma - 1)s^{(pr)}x, c_w(\ttau) - (\ttau - 1)s^{(pr)}x\right) \\
                && = (0, (\tgamma - 1)(w - v), (\ttau - 1)(w - v)) \\
                && = d_0(w - v)
            \end{eqnarray*}
            is clearly in $\im d_0$.
            \item Fix $c'$. Let $x, y \in \Dt^{\dagger, pr}_L(V)$ both represent the class $\eta^{(pr)}([c'\vert_{H_{\infty}}])$. Let $v_x, v_y \in V$ be elements such that
            \[
                c'(h) = (h - 1)s^{(pr)}x + (h - 1)v_x, \quad c'(h) = (h - 1)s^{(pr)}y + (h - 1)v_y.
            \]
            In particular, $s^{(pr)}y + v_y - s^{(pr)}x - v_x \in \Dt^{\dagger, r}_L(V)$. Let $c_x, c_y : \G_K \to V$ be defined by $c_x(g) = c'(g) - (g - 1)v_x$ and $c_y(g) = c'(g) - (g - 1)v_y$. Then, the difference
            \begin{eqnarray*}
                 && \left(-x, c_x(\tgamma) - (\tgamma - 1)s^{(pr)}x, c_x(\ttau) - (\ttau - 1)s^{(pr)}x \right) - \left(-y, c_y(\tgamma) - (\tgamma - 1)s^{(pr)}y, c_y(\ttau) - (\ttau - 1)s^{(pr)}y \right) \\
                && = (y - x, (\tgamma - 1)(v_y - v_x + s^{(pr)}y - s^{(pr)}x), (\ttau - 1)(v_y - v_x + s^{(pr)}y - s^{(pr)}x)) \\
                && = d_0(v_y - v_x + s^{(pr)}y - s^{(pr)}x)
            \end{eqnarray*}
            is in $\im d_0$.

                \item Let $d, e : \G_K \to V$ represent $\alpha$. Therefore, there exists $v \in V$ such that $d(g) - e(g) = (g - 1)v$. Let $x_d, x_e \in \Dt^{\dagger, pr}_L(V)$ represent $\eta^{(pr)}([d\vert_{H_{\infty}}]) = \eta^{(pr)}([e\vert_{H_{\infty}}])$. Let $v_d, v_e \in V$ be such that
                \[
                    d(h) = (h - 1)s^{(pr)}x_d + (h - 1)v_d, \quad e(h) = (h - 1)s^{(pr)}x_e + (h - 1)v_e.
                \]
                In particular, $v - v_d + v_e - s^{(pr)}x_d + s^{(pr)}x_e \in \Dt^{\dagger, r}_L(V)$. Let $c_d, c_e : \G_K \to V$ be defined by $c_d(g) = d(g) - (g - 1)v_d$ and $c_e(g) = e(g) - (g - 1)v_e$. Then, the difference
                \begin{eqnarray*}
                    && \left(-x_d, c_d(\tgamma) - (\tgamma - 1)s^{(pr)}x_d, c_d(\ttau) - (\ttau - 1)s^{(pr)}x_d\right) - \left(-x_e, c_e(\tgamma) - (\tgamma - 1)s^{(pr)}x_e, c_e(\ttau) - (\ttau - 1)s^{(pr)}x_e\right) \\
                    && = \left(x_e - x_d, (\tgamma - 1)(v - v_d + v_e - s^{(pr)}x_d + s^{(pr)}x_e), (\ttau - 1)(v - v_d + v_e - s^{(pr)}x_d + s^{(pr)}x_e)\right) \\
                    && = d_0(v - v_d + v_e - s^{(pr)}x_d + s^{(pr)}x_e)
                \end{eqnarray*}
                is in $\im d_0$.
        \end{itemize}
        Lemma~\ref{Triple is in ker d1} below shows that the image of $h_1$ actually belongs to $\dfrac{\ker d_1}{\im d_0}$. The fact that $h_1$ is an $S$-linear homomorphism is easy to check.
    \end{proof}

    \begin{lemm}\label{Triple is in ker d1}
        With the same assumptions as Lemma~\ref{Definition of h1}, we have
        \begin{enumerate}
            \item $(g - 1)(-x) + (1 - \varphi)(c(g) - (g - 1)s^{(pr)}x) = 0$,
            \item 
                $(\tau - 1)(c(\tgamma) - (\tgamma - 1)s^{(pr)}x) + (1 - \delta^{-1}\gamma)(c(\ttau) - (\ttau - 1)s^{(pr)}x) = 0.$
        \end{enumerate}
        \begin{proof}
            Since the action of $\G_K$ commutes with $\varphi$, we see that $(g - 1)(-x) - (1 - \varphi)(g - 1)s^{(pr)}x = 0$. Also, since $c(g) \in V$, we have 
            $(1 - \varphi)c(g) = 0$. This proves part $1$.

            Next, we prove part $2$. It suffices to show that
            \[
                (\tau^{\chi(\gamma)} - 1)(c(\tgamma) - (\tgamma - 1)s^{(pr)}x) + (\delta - \gamma)(c(\ttau) - (\ttau - 1)s^{(pr)}x) = 0.
            \]
            
            For $n \geq 0$, let $\chi_n(\gamma)$ be a sequence of non-negative integers converging to $\chi(\gamma)$ in the $p$-adic topology. Therefore $\ttau^{\chi_n(\gamma)} \to \ttau^{\chi(\gamma)}$. We prove
            \[
                (\ttau^{\chi(\gamma)} - 1)(c(\tgamma) - (\tgamma - 1)s^{(pr)}x) + \left((1 + \ttau + \cdots + \ttau^{\chi_n(\gamma) - 1}) - \tgamma\right)(c(\ttau) - (\ttau - 1)s^{(pr)}x) \to 0.
            \]
            Using the cocycle condition, we write the last expression as
            \begin{eqnarray*}
                && c(\ttau^{\chi(\gamma)}\tgamma) - c(\ttau^{\chi(\gamma)}) - c(\tgamma) - \ttau^{\chi(\gamma)}\tgamma s^{(pr)}x + \ttau^{\chi(\gamma)} s^{(pr)}x + \tgamma s^{(pr)}x - s^{(pr)}x \\
                && + c(\tilde{\tau}^{\chi_n(\gamma)}) - \tgamma c(\ttau) - \tilde{\tau}^{\chi_n(\gamma)}s^{(pr)}x + s^{(pr)}x + \tgamma\ttau s^{(pr)}x - \tgamma s^{(pr)}x.
            \end{eqnarray*}
            Letting $n \to \infty$, we see that the second term in the first line cancels with the first term in the second line, and the fifth term in the first line cancels with the third term in the second line. Therefore, after taking limits, making some cancellations and applying the cocycle condition, the expression above becomes
            \[
                c(\ttau^{\chi(\gamma)}\tgamma) - c(\tgamma\ttau) - \ttau^{\chi(\gamma)}\tgamma s^{(pr)}x + \tgamma\ttau s^{(pr)}x.
            \]
            Writing $\ttau^{\chi(\gamma)}\tgamma = \tgamma\ttau h$ for some $h \in H_{\infty}$ and applying the cocycle condition again, the expression above becomes
           
            \[
                \tgamma\ttau c(h) - \tgamma\ttau (h - 1)s^{(pr)}x.
            \]
            This is $0$ since $c(h) = (h - 1)s^{(pr)}x$.
        \end{proof}
    \end{lemm}

\subsection{Computations for $H^2(\G_K, V)$} Now that we have defined the map $h_1 : H^1(\G_K, V) \to \frac{\ker d_1}{\im d_0}$, we construct a map $h_2 : H^2(\G_K, V) \to \frac{\ker d_2}{\im d_1}$.

We henceforth make the following assumption for ease of exposition. This will be removed in Remark~\ref{Remove Tors}. 

\vspace{.2cm}

\noindent{\textbf{(Tors)}}:  We assume that $\mu_p \subset K$, so $\langle \gamma \rangle \simeq \Zp$. 
\vspace{.2cm}

This hypothesis implies that we may choose a $\tgamma \in \G_K$ such that $\langle \tgamma \rangle \simeq \Z_p$.

    \begin{lemm}\label{x and y as 2-cocycles}
        The functions $\beta_{\tgamma} : \langle \tgamma \rangle \times \langle \tgamma \rangle \to V$ and $\beta_{\ttau} : \langle \ttau \rangle \times \langle \ttau \rangle \to V$ defined by
        \[
            \beta_{\tgamma}(\tgamma^a, \tgamma^b) = s^{(pr)}\gamma^a \frac{\gamma^b - 1}{\gamma - 1}x - \tgamma^a s^{(pr)}\frac{\gamma^b - 1}{\gamma - 1}x, \quad \beta_{\ttau}(\ttau^a, \ttau^b) = s^{(pr)}\tau^a\frac{\tau^b - 1}{\tau - 1}y - \ttau^a s^{(pr)}\frac{\tau^b - 1}{\tau - 1}y
        \]
        are $2$-cocycles.
    \end{lemm}
    \begin{proof}
        This proof is left as an exercise to the reader.
    \end{proof}
    
    \begin{lemm}\label{cocycle value is 0 if second coordinate in H}
        Every class $\alpha \in H^2(\G_K, V)$ can be represented by a cocycle $c: \G_K \times \G_K \to V$ such that $c(g, h) = 0$ for all $g \in \G_K$, $h \in H_{\infty}$.
    \end{lemm}
    \begin{proof}
        Let $c'$ be a normalized representative for $\alpha$, i.e., $c'(g, 1) = 0 = c'(1, g)$ for all $g \in \G_K$.
 Since $c'\vert_{H_{\infty} \times H_{\infty}}$ represents the zero class in $H^2(H_{\infty}, V)$, we see that there exists $f' : H_{\infty} \to V$ such that $c' + \partial f' = 0$ on $H_{\infty}\times H_{\infty}$. We extend $f' \text{ to } \G_K$ by setting $f'(\tgamma^a\ttau^b) = 0$ for all $a, b \in \Zp$. Further we set $f'(\tgamma^a\ttau^b h)$ as 
\[
            f'(\tgamma^a\ttau^b h) = c'(\tgamma^a\ttau^b, h)  + \tgamma^a\ttau^b f'(h).
        \]
         With $c=c'+\partial f'$, one checks that $c(\tgamma^a\ttau^b, h) = 0$. Hence from the cocycle condition we get that 
          \[
            c(\tgamma^a\ttau^b h_1, h_2) = -c(\tgamma^a\ttau^b, h_1) + c(\tgamma^a\ttau^b, h_1h_2) + \tgamma^a\ttau^b c(h_1, h_2) = 0
        \]
        for any $a, b \in \Zp$ and $h_1, h_2 \in H_{\infty}$.

    \end{proof}

    \begin{lemm}\label{Existence of a c}
        Let $\alpha$ be an element in  $H^2(\G_K, V)$. Choose $x$ and $y \in \Dt^{\dagger, pr}_L(V)$ congruent to $\eta^{(pr)}\Big(\rho(\alpha)(\gamma)\Big)$ and $ \eta^{(pr)}\Big(\rho(\alpha)(\tau)\Big)$ respectively, modulo $(\varphi - 1)\Dt^{\dagger, r}_L(V)$. Then  $\alpha$ can be represented by a cocycle $c : \G_K \times \G_K \to V$ such that
        \begin{enumerate}
            \item $c(g, h) = 0$ for all $g \in \G_K, h \in H_{\infty}$,
            \item $c(h, \tgamma) = (1 - h)s^{(pr)}x$, $c(h, \ttau) = (1 - h)s^{(pr)}y$,
            \item $c(\tgamma^a, \tgamma) = s^{(pr)}\gamma^a x - \tgamma^a s^{(pr)}x$ for all $a \in \Zp$,
            \item $c(\ttau^a, \ttau) = s^{(pr)}\tau^a y - \ttau^a s^{(pr)}y$ for all $a \in \Zp$,
            \item $c(\tgamma^a, \ttau^b) = s^{(pr)}\gamma^a \frac{\tau^b - 1}{\tau - 1}y - \tgamma^a s^{(pr)}\frac{\tau^b - 1}{\tau - 1}y$ for all $a, b \in \Zp$.
        \end{enumerate}
    \end{lemm}
    \begin{proof}
        Using Lemma~\ref{cocycle value is 0 if second coordinate in H}, we pick $c'' : \G_K \times \G_K \to V$ such that $(1)$ is satisfied. 
 Therefore there exist $v_1, v_2 \in V$ such that: \begin{itemize}
        \item $
            c''(\tgamma, \tgamma^{-1}h\tgamma) - c''(h, \tgamma) = (h - 1)s^{(pr)}x + (h - 1)v_1, $ \item $ c''(\ttau, \ttau^{-1}h\ttau) - c''(h, \ttau) = (h - 1)s^{(pr)}y + (h - 1)v_2.
        $
        \end{itemize}
        Let $f'' : \G_K  \to V$ be a continuous function factoring through $G_{\infty}$ and satisfying  $f''(\tgamma) = v_1$, $f''(\ttau) = v_2$, and $f''(1)=0$. Then $c' := c'' + \partial f''$ satisfies $(1), (2)$.

        Since $\langle \tgamma \rangle , \langle \ttau \rangle \subseteq \G_K$ are procyclic, we see that $H^2(\langle \tgamma \rangle, V) = 0 = H^2(\langle \ttau \rangle, V)$. Therefore using Lemma~\ref{x and y as 2-cocycles}, we get continuous maps $f'_{\tgamma}: \langle \tgamma \rangle \to V$ and $f'_{\ttau} : \langle \ttau \rangle \to V$ such that
        \begin{equation}\label{eq:1}
           c'\vert_{\langle{\tgamma \rangle} \times \langle \tgamma \rangle} = \beta_{\tgamma} - \partial f'_{\tgamma}, \quad c'\vert_{\langle \ttau \rangle \times \langle \ttau \rangle} = \beta_{\ttau} - \partial f'_{\ttau}.  
        \end{equation} Modifying $f'_{\tgamma}$ and $f'_{\ttau}$ by a $1$-cocycle, we may assume that  $f'_{\tgamma}(\tgamma) = 0 = f'_{\ttau}(\ttau)$.

        We define $f' : \G_K \to G_{\infty} \to V$ by
        \[
            f'(\tgamma^a\ttau^b) = c'(\tgamma^a, \ttau^b) + f'_{\tgamma}(\tgamma^a) + \tgamma^a f_{\ttau}'(\ttau^b) - s^{(pr)}\gamma^a \frac{\tau^b - 1}{\tau - 1}y + \tgamma^a s^{(pr)}\frac{\tau^b - 1}{\tau - 1}y.
        \]
        Using \eqref{eq:1}, we may check that $f'_{\ttau}(1)=0=f'_{\tgamma}(1).$ Therefore $f'(\tgamma^a)=f'_{\tgamma}(\tgamma^a)$ and $f'(\ttau^b)=f'_{\ttau}(\ttau^b)$.
        Defining $c = c' + \partial f'$, we can check that $c$ satisfies $(3), (4) \text{ and } (5)$. Next, we verify that $c$ also satisfies $(1) \text{ and } (2)$. Indeed, 
        \begin{enumerate}[(i)]
            \item $c(g, h) = c'(g, h) + f'(g) - f'(gh) + gf'(h) = 0 $ (since $f'$ factors through $G_\infty$),
            \item $c(h, \tgamma) = c'(h, \tgamma) + f'(h) - f'(h\tgamma) + hf'(\tgamma) = (1 - h)s^{(pr)}x$ and similarly $c(h, \ttau) = (1 - h)s^{(pr)}y$. \qedhere
        \end{enumerate}
    \end{proof}

    \begin{lemm}\label{Second coordinate only matters mod H}
        Let $c$ be a $2$-cocycle such that $c(g, h) = 0$ for any $g \in \G_K, h \in H_{\infty}$. Then, $c(g_1, g_2 h) = c(g_1, g_2)$ for any $g_1, g_2 \in \G_K$ and $h \in H_{\infty}$.
    \end{lemm}
    \begin{proof}
        This proof is left as an exercise to the reader.
    \end{proof}
In Lemma \ref{Existence of a c}, we have obtained values of $c(h,\tgamma)$ and $c(h,\ttau)$. The following lemma computes the values of $c(h,g)$ for an arbitrary $g \in \G_K$. 
    \begin{lemm}\label{lem:c(h,g)}
        Let $c$ be a $2$-cocycle obtained in Lemma~\ref{Existence of a c}. Then, for any $g = \tgamma^a\ttau^b h_0 \in \G_K$, we have
        \[
            c(h, g) = -(h - 1)s^{(pr)}\left(\frac{\gamma^a - 1}{\gamma - 1}x + \gamma^a\frac{\tau^b - 1}{\tau - 1}y\right).
        \]
    \end{lemm}
    \begin{proof}
        For an integer $k \geq 0$, we have
        \begin{eqnarray*}
            (h - 1)\ttau^k s^{(pr)}y & = & \ttau^{k}(\ttau^{-k} h \ttau^{k} - 1)s^{(pr)}y \\
            & = & -\ttau^k c(\ttau^{-k}h \ttau^{k}, \ttau) \\
            & = & c(\ttau^k, \ttau^{-k}h\ttau^{k + 1}) - c(h\ttau^k, \ttau) - c(\ttau^k, \ttau^{-k}h \ttau^k) \\
            & = & c(\ttau^k, \ttau) - c(h\ttau^k, \ttau).
        \end{eqnarray*}
       The second and the fourth equalities follow from Lemma \ref{Existence of a c} and Lemma \ref{Second coordinate only matters mod H},  while the third equality is the cocycle condition.  Using the $2$-cocycle condition on $c(h \ttau^k, \ttau)$, we see that
        \[
            (h - 1)\ttau^k s^{(pr)}y = (1 - h)c(\ttau^k, \ttau) + c(h, \ttau^k) - c(h, \ttau^{k + 1}).
        \]
        Using Lemma~\ref{Existence of a c}(4), we see that
        \[
            (h - 1)s^{(pr)}\tau^k y = c(h, \ttau^k) - c(h, \ttau^{k + 1}).
        \]
        Given any integer $b \geq 1$, we may sum the display above for $k = 0, \cdots, b - 1$ to get
        \begin{eqnarray}\label{tau part of c(h, _)}
            -c(h, \ttau^b) = (h - 1)s^{(pr)}\frac{\tau^b - 1}{\tau - 1}y.
        \end{eqnarray}
        We can similarly prove that
        \begin{eqnarray}\label{gamma part of c(h, _)}
            -c(h, \tgamma^a) = (h - 1)s^{(pr)}\frac{\gamma^a - 1}{\gamma - 1}x.
        \end{eqnarray}

        Now,
        \begin{eqnarray*}
            (h - 1)\tgamma^a s^{(pr)}\frac{\tau^b - 1}{\tau - 1}y & = & \tgamma^a (\tgamma^{-a}h\tgamma^a - 1)s^{(pr)}\frac{\tau^b - 1}{\tau - 1}y \\
            & = & -\tgamma^a c(\tgamma^{-a}h\tgamma^{a}, \ttau^b) \\
            & = & c(\tgamma^a, \tgamma^{-a} h \tgamma^a \ttau^b) - c(h\tgamma^a, \ttau^b) - c(\tgamma^a, \tgamma^{-a}h \tgamma^a) \\
            & = &   (1 - h)c(\tgamma^a, \ttau^b) + c(h, \tgamma^a) - c(h, \tgamma^a\ttau^b).
        \end{eqnarray*}
      Here, the second equality follows from \eqref{tau part of c(h, _)} and the fourth equality from Lemmas \ref{Existence of a c}, \ref{Second coordinate only matters mod H} and the cocycle condition applied to $c(h\tgamma^a,\ttau^b)$.
      Using Lemma~\ref{Existence of a c}(5) and \eqref{gamma part of c(h, _)}, we see that
        \[
            -c(h, \tgamma^a\ttau^b) = (h - 1)s^{(pr)}\left(\frac{\gamma^a - 1}{\gamma - 1}x + \gamma^a\frac{\tau^b - 1}{\tau - 1}y\right).
        \]

        This lemma is now proved by applying Lemma~\ref{Second coordinate only matters mod H} to the display above.
    \end{proof}

    \begin{prop}\label{H2 fixed under Hinfty}
        Let $\alpha \in H^2(\G_K, V)$. Choose $x, y \in \Dt^{\dagger, pr}_L(V)$ congruent to $\eta^{(pr)}\Big(\rho(\alpha)(\gamma)\Big), \eta^{(pr)}\Big(\rho(\alpha)(\tau)\Big)$, respectively, modulo $(\varphi - 1)\Dt^{\dagger, r}_L(V)$. Then, for $c : \G_K \times \G_K \to V$ obtained using Lemma~\ref{Existence of a c}, the element $z := c(\tgamma, \tgamma^{-1}\ttau\tgamma) - c(\ttau, \tgamma) - (\ttau - 1)s^{(pr)}x + \tgamma s^{(pr)}\left(\frac{\tau^{\chi(\gamma)^{-1}} - 1}{\tau - 1}y\right) - s^{(pr)}y \in  (S \hat{\otimes}_{\Qp} \Bt^{\dagger, r}) \otimes_S V$ belongs to $\Dt^{\dagger, r}_L(V)$, i.e., it is $H_{\infty}$-fixed.
        Furthermore, the map
        \begin{eqnarray*}
            h_2 : H^2(\G_K, V) & \to & \frac{\ker d_2}{\im d_1} \\
            \alpha & \mapsto & \Big(-x, -y, -z\Big) \mod \im d_1
        \end{eqnarray*}
        is independent of the choices of $x$, $y$ and $c$, and is an $S$-linear homomorphism.
    \end{prop}
    \begin{proof}
       For $g_1 = \tgamma^{a_1}\ttau^{b_1}h_1, g_2 = \tgamma^{a_2}\ttau^{b_2}h_2 \in \G_K$, define the element $z_{g_1, g_2}$ as 
        \begin{eqnarray*}
            z_{g_1, g_2} & = & c(g_1, g_2) + s^{(pr)}\left(\frac{\gamma^{a_1} - 1}{\gamma - 1}x + \gamma^{a_1}\frac{\tau^{b_1} - 1}{\tau - 1}y\right) - s^{(pr)}\left(\frac{\gamma^{a_1 + a_2} - 1}{\gamma - 1} x + \gamma^{a_1 + a_2}\frac{\tau^{b_1 \chi(\gamma)^{-a_2}+b_2} - 1}{\tau - 1}y\right)  \\
            & & \qquad \qquad \qquad \qquad \qquad + g_1s^{(pr)}\left(\frac{\gamma^{a_2} - 1}{\gamma - 1}x + \gamma^{a_2}\frac{\tau^{b_2} - 1}{\tau - 1}y\right).
        \end{eqnarray*}
        We claim that $z_{g_1, g_2}$ is fixed under $H_{\infty}$. Indeed, using Lemma \ref{lem:c(h,g)}, we obtain
        \begin{eqnarray*}
            (h - 1)z_{g_1, g_2} & = & c(hg_1, g_2) + c(h, g_1) - c(h, g_1g_2) - c(g_1, g_2) - c(h, g_1) + c(h, g_1g_2) - g_1 c(g_1^{-1}hg_1, g_2) \\
            & = & c(hg_1, g_2) - c(g_1, g_2) - c(hg_1, g_2) - c(g_1, g_1^{-1}hg_1) + c(g_1, g_1^{-1}hg_1g_2).
        \end{eqnarray*}
        The first and the third terms cancel, the second term cancels with the fifth term using Lemma~\ref{Second coordinate only matters mod H} and the fourth term is $0$ by Lemma~\ref{Existence of a c}(1). Therefore
        \begin{eqnarray}\label{Differences are H-fixed}
            (h - 1)z_{g_1, g_2} = 0.
        \end{eqnarray}

       This computation shows that $z = z_{g_1, g_2} - z_{g_1', g_2'}$ for $(g_1, g_2) = (\tgamma, \tgamma^{-1}\ttau\tgamma)$ and $(g_1', g_2') = (\ttau, \tgamma)$ is fixed under $H_{\infty}$.

The fact that $h_2$ is independent of the choices of $x, y$ and $c$ will be proved in Section \ref{sec:well-defined}.
\end{proof}

    \begin{rema}\label{Remove Tors}
        In this remark, we show how to remove the assumption \textbf{(Tors)}. This makes the following theorem true for arbitrary $K$. 
        
        Assume that $K$ is arbitrary. So, $\Gamma_K \simeq \Delta \times \Zp$ for a finite prime-to-$p$ subgroup $\Delta$ of $\Gamma_K$. Let $\gamma$ be a topological generator of $\Gamma_K$. Choose an arbitrary lift $\tgamma \in H_{\tau, K}$ of $\gamma$. So the procyclic subgroup of $H_{\tau, K}$ generated by $\tgamma$ surjects onto $\Gamma_K$. Choose generators $d$ of $\Delta$ and $\gamma'$ of the Sylow-$p$-subgroup of $\Gamma_K$ such that $\gamma = d\gamma'$. Choosing lifts $\tilde{d}$ and $\tilde{\gamma'}$ of $d$ and $\gamma$, respectively in $\langle \tgamma\rangle$, we may define $\tilde{\sigma} = \tilde{d}^{a_1}\tilde{\gamma'}^{a_2}$ for any $\sigma = d^{a_1}\gamma'^{a_2} \in \Gamma_K$. Note that the lift $\tilde{\gamma'}$ can be chosen to be a multiplicative lift.

        If, in the lemmas above, one replaces all $\tgamma^a$ with $\tilde{d}^{a_1}\tilde{\gamma'}^{a_2}$, where $a_1 \in \{0, 1, \ldots, |\Delta| - 1\}$ and $a_2 \in \Zp$, then the proofs go through mutatis mutandis.
    \end{rema}

    \begin{theo}\label{Overconvergent theorem}
        Let $V$ be a family of representations of $\G_K$ and $\D^{\dagger, r}_{\tau, K}(V)$ the associated $(\varphi, \tau)$-module. For $r \geq r_0$, let $\D^{r}_L = S\hat{\otimes}\Bt^{\dagger, r}_L\otimes_{S \hat{\otimes}\B_{\tau, K}^{\dagger, r}} D_{\tau, K}^{\dagger, r}(V)$. Then, the cohomology of the complex
        {\small
        \[
            0 \to \D^{r}_L \xrightarrow{d_0 = \begin{pmatrix}\varphi - 1 \\ \gamma - 1 \\ \tau - 1\end{pmatrix}} \D^{pr}_L \oplus \D^{r}_L \oplus \D^{r}_L \xrightarrow{d_1 = \begin{pmatrix}\gamma - 1 & 1 - \varphi & 0 \\ \tau - 1 & 0 & 1 - \varphi \\ 0 & \tau - 1 & 1 - \delta^{-1}\gamma\end{pmatrix}} \D^{pr}_L \oplus \D^{pr}_L \oplus \D^{r}_L \xrightarrow{d_2 = (\tau - 1, 1 - \delta^{-1}\gamma, \varphi - 1)} \D^{pr}_L \to 0
        \]
        }
        is isomorphic to the Galois cohomology of $V$, where $\delta = \frac{\tau^{\chi(\gamma)} - 1}{\tau - 1}$.
    \end{theo}
    \begin{proof}
        The Hochschild-Serre spectral sequence gives the following exact sequence
        \begin{eqnarray}\label{Plain 6 term HS}
            && 0 \to H^1(G_{\infty}, V^{H_{\infty}}) \to H^1(\G_K, V) \to H^1(H_{\infty}, V)^{G_{\infty}} \to  \\
            && \qquad \qquad H^2(G_{\infty}, V^{H_{\infty}}) \to H^2(\G_K, V) \to H^1(G_\infty, H^1(H_{\infty}, V)) \to 0. \nonumber
        \end{eqnarray}
        Next, we produce a map of complexes between sequence~\eqref{Plain 6 term HS} and the sequence in Proposition~\eqref{Substitute for 6 term HS}:
        \begin{equation}\label{eq:comparing}
            \begin{tikzcd}[column sep = small]
                H^1(G_{\infty}, V^{H_{\infty}}) \ar[r, "\inf", hookrightarrow]\ar[d, "{(\mathrm{ev}_{\gamma}, \mathrm{ev}_{\tau})}"] & H^1(\G_K, V) \ar[r, "\mathrm{res}"]\ar[d, "h_1"] & H^1(H_{\infty}, V)^{G_{\infty}} \ar[r, "\mathrm{tr}"] \ar[d, "-\eta^{(pr)}"] & H^2(G_{\infty}, V^{H_{\infty}}) \ar[r, "\inf"] \ar[d, "{\mathrm{ev}_{\left(\tau, \gamma\right)}} - \mathrm{ev}_{\left(\gamma, \gamma^{-1}\tau\gamma\right)}"] & H^2(\G_K, V) \ar[r, "\rho", twoheadrightarrow] \ar[d, "h_2"] & H^1(G_\infty, H^1(H_{\infty}, V)) \ar[d, "{(-\mathrm{ev}_{\gamma}, -\mathrm{ev}_{\tau})}"] \\
                \dfrac{\ker g_{(\D_L^{r})^{\varphi = 1}}}{\im f_{(\D_L^{r})^{\varphi = 1}}} \ar[r, "\delta_1", hookrightarrow] & \dfrac{\ker d_1}{\im d_0} \ar[r, "\delta_2"] & \left(\frac{\D_L^{pr}}{(\varphi - 1)}\right)^{\tau = 1, \gamma = 1} \ar[r, "\delta_3"] & \frac{(\D_L^{r})^{\varphi = 1}}{(\tau - 1, 1 - \delta^{-1}\gamma)} \ar[r, "\delta_4"] & \dfrac{\ker d_2}{\im d_1} \ar[r, "\delta_5", twoheadrightarrow] & \dfrac{\ker g_{(\D_L^{pr}/(\varphi - 1))}}{\im f_{(\D_L^{pr}/(\varphi - 1))}},
            \end{tikzcd}
       \end{equation}
        where $\mathrm{ev}_{*}$ means evaluating a function at $*$.

        Next, we show that the diagram above commutes.
        \begin{enumerate}
            \item Let $\alpha \in H^1(G_{\infty}, V^{H_{\infty}})$. Choose a $1$-cocycle $c : G_{\infty} \to V^{H_{\infty}}$ representing $\alpha$. Since $\mathrm{res}\circ\inf(\alpha) = 0$ in $H^1(H_{\infty}, V)$, we may choose $x$ in Proposition~\ref{Definition of h1} to be $0$. Then,
            \[
                h_1\circ\inf(\alpha) = (0, c(\gamma), c(\tau)).
            \]
            This is precisely equal to $\delta_1 \circ (\mathrm{ev}_\gamma, \mathrm{ev}_{\tau}) (\alpha)$.

            \item Let $\alpha \in H^1(\G_K, V)$. Then,
            \[
                \delta_2\circ h_1(\alpha) = -x,
            \]
            where $x \in \D^{pr}_L$ is congruent to $\eta^{(pr)}(\mathrm{res}(\alpha))$ modulo $(\varphi - 1)\D^{r}_L$. This is obviously equal to $-\eta^{(pr)}\circ\mathrm{res}(\alpha)$.

            \item Let $\alpha \in H^1(H_{\infty}, V)^{G_{\infty}}$. Let $c : H_{\infty} \to V$ be a $1$-cocycle representing $\alpha$. Fix $x \in \D_L^{pr}$ such that $\eta^{(pr)}(\alpha) = x$ modulo $(\varphi - 1)\D_L^{r}$. We may assume that $c$ is given by $c(h) = (h - 1)s^{(pr)}x$ for all $h \in H_{\infty}$.
            
            As $\alpha$ is fixed by $G_{\infty}$, there are $v, w \in V$ such that
            \[
                \tgamma c(\tgamma^{-1}h\tgamma) = c(h) + (h - 1)v \quad \text{and} \quad \ttau c(\ttau^{-1}h\ttau) = c(h) + (h - 1)w
            \]
            for all $h \in H_{\infty}$. These imply that, with 
            $y=\tgamma s^{(pr)}x-s^{(pr)}x-v \in \D_L^{r}$ and $z =\ttau s^{(pr)}x-s^{(pr)}x-w \in \D_L^{r}$, we have
            \begin{eqnarray}\label{Prep to define delta_3}
                (\tgamma - 1)s^{(pr)}x = y + v \quad \text{ and } \quad (\ttau - 1)s^{(pr)}x = z + w.
            \end{eqnarray}
            In particular, after applying $\varphi - 1$, we get
            \[
                (\gamma - 1)x = (\varphi - 1)y \quad \text{ and } \quad (\tau - 1)x = (\varphi - 1)z.
            \]

           \noindent \textbf{Claim:} For arbitrary integers $a, b \geq 0$, the element
        
    \begin{eqnarray}\label{Extension to G_K}
                (\tgamma^a \ttau^b - 1)s^{(pr)}x - \frac{\gamma^a - 1}{\gamma - 1}y - \gamma^a \frac{\tau^b - 1}{\tau - 1}z \in V.
            \end{eqnarray}
           \noindent \textit{Proof of Claim:} Consider $(\ttau - 1)s^{(pr)}x = z + w$. Applying $\ttau^{m - 1}$ and iteratively substituting, we see that
            \begin{eqnarray}\label{Extension to powers of tau}
                (\ttau^b - 1)s^{(pr)}x = (1 + \tau + \cdots + \tau^{b - 1})z + w' = \frac{\tau^b - 1}{\tau - 1}z + w'
            \end{eqnarray}
            for some $w' \in V$.  Similarly, we see that
            \begin{eqnarray}\label{Extension to powers of gamma}
                (\tgamma^a - 1)s^{(pr)}x = \frac{\gamma^a - 1}{\gamma - 1}y + v'
            \end{eqnarray}
            for some $v' \in V$. Now we apply $\tgamma^a$ to \eqref{Extension to powers of tau} and substitute \eqref{Extension to powers of gamma} to get
            \[
                (\tgamma^a\ttau^b - 1)s^{(pr)}x = \frac{\gamma^a - 1}{\gamma - 1}y + \gamma^a\frac{\tau^b - 1}{\tau - 1}z + v''
            \]
            for some $v'' \in V$. This proves the claim.

            A quick check using continuity shows that formula \eqref{Extension to G_K} is true for arbitrary $a, b \in \Zp$. This means that we have extended $c$ to a continuous $1$-cochain $c' : \G_K \to V$ given by 
            $$c'(g):=(g-1)s^{(pr)}x - \frac{\gamma^a - 1}{\gamma - 1}y - \gamma^a \frac{\tau^b - 1}{\tau - 1}z $$
            where $g\mapsto \gamma^a\tau^b$ under the canonical surjection $\G_K \rightarrow G_\infty$. We can easily check the following:  
            \begin{enumerate}
                \item $c'\vert_{H_\infty}=c$,
                \item $c'(st)=c'(s)+sc'(t)$ for all $s \in \G_K$ and $t \in H_\infty$,
                 \item $c'(ts)=c'(t)+tc'(s)$ for all $s \in \G_K$ and $t \in H_\infty.$
            \end{enumerate}
            A formal check done, e.g., in \cite[Proposition 1.6.6]{Neu}, shows that $\partial c'$ satisfies the three bullet points given in Section \ref{Explicit maps in inflation-restriction}.
            The transgression map $\mathrm{tr}$ applied to $\alpha$ yields the class $[\partial c']$. Finally, we check that $$(\mathrm{ev}_{(\gamma, \gamma^{-1}\tau\gamma)}-\mathrm{ev}_{(\tau, \gamma)})\circ \mathrm{tr} = \delta_3 \circ \eta^{(pr)}.$$ Hence we are reduced to show that 
            \[
                \partial c'(\gamma, \gamma^{-1}\tau\gamma) - \partial c'(\tau, \gamma) = (\tau - 1)y + (1 - \delta^{-1}\gamma)z.
            \]
            Indeed, the left hand side 
            \begin{eqnarray*}
                & = & c'(\tgamma) - c'(\tgamma\tilde{\tau^{\chi(\gamma)^{-1}}}) + \tgamma c'(\tilde{\tau^{\chi(\gamma)^{-1}}}) - c'(\ttau) + c'(\ttau\tgamma) - \ttau c'(\tgamma) \\
                & = & (1 - \ttau)[(\tgamma - 1)s^{(pr)}x - y] + \tgamma\Big[(\tilde{\tau^{\chi(\gamma)^{-1}}} - 1)s^{(pr)}x - \frac{\tau^{\chi(\gamma)^{-1}} - 1}{\tau - 1}z\Big] - [(\ttau - 1)s^{(pr)}x - z] \\
                & & \quad - \tilde{\tau\gamma}(h_1 - 1)s^{(pr)}x + \tilde{\tau\gamma}(h_2 - 1)s^{(pr)}x,
            \end{eqnarray*}
            where $h_1, h_2 \in H_{\infty}$ are defined by $\tgamma\tilde{\tau^{\chi(\gamma)^{-1}}} = \tilde{\tau\gamma}h_1$ and $\ttau\tgamma = \tilde{\tau\gamma}h_2$. After some cancellations, we get that the expression above is $ (\tau - 1)y + (1 - \delta^{-1}\gamma)z$ which completes the proof of the commutativity of the third square.

            \item Let $\alpha \in H^2(G_{\infty}, V^{H_{\infty}})$. Pick a normalized cocycle $c' : G_{\infty}\times G_{\infty} \to V^{H_{\infty}}$ representing $\alpha$. Pre-composing $c'$ with  the canonical surjection $\G_K \to G_{\infty}$ represents the class $\inf \alpha$. We may modify $c'$ by a coboundary so that the new cocycle $c$ satisfies all conditions in Lemma~\ref{Existence of a c} for $x = 0 = y$. Indeed, we construct $f' : \G_K \to V^{H_{\infty}}$ as is done in the proof of Lemma~\ref{Existence of a c} (the fact that $f'(\G_K) \subseteq V^{H_{\infty}}$ can be seen by evaluating \eqref{eq:1} at $(\tgamma^a,\tgamma)$ and $(\ttau^b,\ttau)$ and applying induction). Thinking of $f'$ as a continuous function from $G_{\infty}$ to $V^{H_{\infty}}$, we get a new $2$-cocycle $c = c' + \partial f'$ representing $\alpha$ such that pre-composing $c$ with $\G_K \to G_{\infty}$ satisfies all conditions in Lemma~\ref{Existence of a c} for $x = 0 = y$. Then,
            \[
                h_2 \circ \inf(\alpha) = (0, 0, c(\ttau, \tgamma) - c(\tgamma, \tgamma^{-1}\ttau\tgamma)) = \delta_4 \circ (\mathrm{ev_{(\tau, \gamma)}} - \mathrm{ev}_{(\gamma, \gamma^{-1}\tau\gamma)})(\alpha).
            \]

            \item The commutativity of the fifth square is similar to that of the second square.
        \end{enumerate}
        Applying the five lemma twice to diagram \ref{eq:comparing} and noting that all the vertical arrows, except possibly the maps $h_1$ and $h_2$, are isomorphisms by Propositions \ref{Explicit Ginfty cohomology} and \ref{Hinfty cohomology}, we see that the maps $h_1$ and $h_2$ are indeed isomorphisms.
        
\vspace{.2cm}

Since $H^i(H_\infty,V)=0$ for $i \geq 2$ by Proposition \ref{Hinfty cohomology}, we obtain $H^3(\G_K,V) \simeq H^2(G_\infty, H^1(H_\infty,V))$ by Hochschild-Serre.
This module is further isomorphic to 
$ \frac{H^1(H_\infty,V)}{\im g_{H^1(H_\infty,V)}}$
by Proposition \ref{Explicit Ginfty cohomology} which equals
$\frac{\D_L^{pr}}{\im d_2}$ by Proposition \ref{Hinfty cohomology}.
    \end{proof}

Taking direct limits in Theorem \ref{Overconvergent theorem}, we get the following corollary.

\begin{coro}
    Let $V$ be a family of representations of $\G_K$ and $\D_L := \Dt^{\dagger}_{L}(V)$ as in Definition~\ref{Definition of families of phi tau modules and their base change}. Then, the cohomology of the complex
        {\small
        \[
            0 \to \D_L \xrightarrow{d_0 = \begin{pmatrix}\varphi - 1 \\ \gamma - 1 \\ \tau - 1\end{pmatrix}} \D_L \oplus \D_L \oplus \D_L \xrightarrow{d_1 = \begin{pmatrix}\gamma - 1 & 1 - \varphi & 0 \\ \tau - 1 & 0 & 1 - \varphi \\ 0 & \tau - 1 & 1 - \delta^{-1}\gamma\end{pmatrix}} \D_L \oplus \D_L \oplus \D_L \xrightarrow{d_2 = (\tau - 1, 1 - \delta^{-1}\gamma, \varphi - 1)} \D_L \to 0
        \]
        }
        is isomorphic to the Galois cohomology of $V$, where $\delta = \frac{\tau^{\chi(\gamma)} - 1}{\tau - 1}$.
\end{coro}

\subsection{Completion of the proof of Proposition \ref{H2 fixed under Hinfty}}\label{sec:well-defined}
In this section we are going to show that the map $h_2$ in Proposition \ref{H2 fixed under Hinfty} is well-defined. 

Let us fix $x, y \in \Dt^{\dagger, pr}_L(V)$ congruent to $\eta^{(pr)}\Big(\rho(\alpha)(\gamma)\Big), \eta^{(pr)}\Big(\rho(\alpha)(\tau)\Big)$, respectively, modulo $(\varphi - 1)\Dt^{\dagger, r}_L(V)$.  Let $c_1$ and $c_2 = c_1 + \partial f$ be two representatives for $c$ satisfying all the conditions  of Lemma~\ref{Existence of a c}. Then, $\partial f(g, h) = 0 = \partial f(h, g)$ for all $g \in \G_K, h \in H_{\infty}$. Also, by \eqref{Differences are H-fixed}, $\partial f(g_1,g_2) \in V^{H_{\infty}}$ for all $g_1,g_2 \in \G_K$. A proof similar to that of Lemma~\ref{Second coordinate only matters mod H} shows that the vanishing of $\partial f (h,g)$ implies that 
$\partial f (hg_1,g_2)=h\partial f (g_1,g_2)=\partial f (g_1,g_2)$ which gives that 
$\partial f : G_{\infty} \times G_\infty \to V^{H_{\infty}}$ is a $2$-cocycle, i.e., $\partial f \in H^2(G_{\infty},V^{H_{\infty}})$. 
Let $z_1,z_2$ be the elements obtained using Proposition \ref{H2 fixed under Hinfty} with cocyles $c_1$, $c_2$ respectively. From diagram \ref{eq:comparing}, one can see that $$(-x,-y,-z_1)-(-x,-y,-z_2)=\delta_4 \circ \left(\partial f (\gamma,\gamma^{-1}\tau\gamma)-\partial f(\tau,\gamma)\right).$$
But $\inf(\partial f)=0$ implies $\partial f = \mathrm{tr}(f')$ for some $f' \in H^1(H_\infty,V)^{G_\infty}.$ Therefore, by commutativity of the third square in diagram \ref{eq:comparing}, we obtain that $(-x,-y,-z_1)-(-x,-y,-z_2)=\delta_4 \circ \delta_3 \circ (-\eta^{(pr)})(f')=0.$

\vspace{.2cm}

Next we prove independence of choices of $x$ and $y$. Let $x_1, x_2 \in \Dt^{\dagger, pr}_L(V)$ be congruent to $\eta^{(pr)}\Big(\rho(\alpha)(\gamma)\Big)$ modulo $(\varphi - 1)\Dt^{\dagger, r}_L(V)$ and let $y_1, y_2 \in \Dt^{\dagger, pr}_L(V)$ be congruent to $\eta^{(pr)}\Big(\rho(\alpha)(\tau)\Big)$ modulo $(\varphi - 1)\Dt^{\dagger, r}_L(V)$. In particular, there exist $x, y \in \Dt^{\dagger, r}_L(V)$ such that
\[
    x_1 - x_2 = (\varphi - 1)x \text{ and } y_1 - y_2 = (\varphi - 1)y.
\]

Let $c_1 : \G_K \times \G_K \to V$ be a representative of $\alpha$ satisfying all five conditions in Lemma~\ref{Existence of a c} for $x_1$ and $y_1$. Let $f$ be a continuous function on $\G_K$ factoring through $G_{\infty}$ defined by
\[
    f(\tgamma^a\ttau^b) = s^{(pr)}\left[\frac{\gamma^a - 1}{\gamma - 1}(\varphi - 1)x + \gamma^a \frac{\tau^b - 1}{\tau - 1}(\varphi - 1)y\right] - \left[\frac{\gamma^a - 1}{\gamma - 1}x + \gamma^a\frac{\tau^b - 1}{\tau - 1}y\right].
\]
We claim that the cocycle $c_2 := c_1 + \partial f$ satisfies the conditions of Lemma~\ref{Existence of a c} for $x_2$ and $y_2$.
\begin{itemize}
    \item Given any $g \in \G_K$ and $h \in H_{\infty}$, we see that
    \[
        c_2(g, h) = c_1(g, h) + f(g) - f(gh) + gf(h) = 0.
    \]
    
    \item Given any $h \in H_{\infty}$, we have
    \begin{eqnarray*}
        c_2(h, \tgamma) & = & c_1(h, \tgamma) + f(h) - f(h\tgamma) + hf(\tgamma) \\
        & = & (1 - h)s^{(pr)}x_1 - (1 - h)f(\tgamma) \\
        & = & (1 - h)\left(s^{(pr)}x_1 - s^{(pr)}(\varphi - 1)x + x\right) \\
        & = & (1 - h)s^{(pr)}x_2.
    \end{eqnarray*}
    Similarly,
$c_2(h, \ttau) = (1 - h)s^{(pr)}y_2.$
    
    \item For any $a \in \Zp$, we have
    \begin{eqnarray*}
        c_2(\tgamma^a, \tgamma) & = & c_1(\tgamma^a, \tgamma) + f(\tgamma^a) - f(\tgamma^{a + 1}) + \tgamma^a f(\tgamma) \\
        & = & s^{(pr)}\gamma^a x_1 - \tgamma^a s^{(pr)}x_1 + s^{(pr)}\left[\frac{\gamma^a - 1}{\gamma - 1}(\varphi - 1)x\right] - \frac{\gamma^a - 1}{\gamma - 1}x \\
        & & \qquad - s^{(pr)}\left[\frac{\gamma^{a + 1} - 1}{\gamma - 1}(\varphi - 1)x\right] + \frac{\gamma^{a + 1} - 1}{\gamma - 1}x + \tgamma^a\left\{s^{(pr)}\left[(\varphi - 1)x\right] - x\right\} \\
        & = & s^{(pr)}\gamma^a x_2 - \tgamma^a s^{(pr)}x_2.
    \end{eqnarray*}
    \item For any $b \in \Zp$, we have
    \begin{eqnarray*}
        c_2(\ttau^b, \ttau) & = & c_1(\ttau^b, \ttau) + f(\ttau^b) - f(\ttau^{b + 1}) + \ttau^b f(\ttau) \\
        & = & s^{(pr)}\tau^b y_1 - \ttau^b s^{(pr)}y_1 + s^{(pr)}\left[\frac{\tau^b - 1}{\tau - 1}(\varphi - 1)y\right] - \frac{\tau^b - 1}{\tau - 1}y \\
        & & \qquad - s^{(pr)}\left[\frac{\tau^{b + 1} - 1}{\tau - 1}(\varphi - 1)y\right] + \frac{\tau^{b + 1} - 1}{\tau - 1}y + \ttau^b\left\{s^{(pr)}\left[(\varphi - 1)y\right] - y\right\} \\
        & = & s^{(pr)}\tau^b y_2 - \ttau^b s^{(pr)}y_2.
    \end{eqnarray*}
    \item For any $a, b \in \Zp$, we have
    \begin{eqnarray*}
        c_2(\tgamma^a, \ttau^b) & = & c_1(\tgamma^a, \ttau^b) + f(\tgamma^a) - f(\tgamma^a \ttau^b) + \tgamma^a f(\ttau^b) \\
        & = & s^{(pr)}\gamma^a\frac{\tau^b - 1}{\tau - 1}y_1 - \tgamma^a s^{(pr)}\frac{\tau^b - 1}{\tau - 1}y_1 + s^{(pr)}\left[\frac{\gamma^a - 1}{\gamma - 1}(\varphi - 1)x\right] - \frac{\gamma^a - 1}{\gamma - 1}x \\
        & & \quad - s^{(pr)}\left[\frac{\gamma^a - 1}{\gamma - 1}(\varphi - 1)x + \gamma^a \frac{\tau^b - 1}{\tau - 1}(\varphi - 1)y\right] + \left[\frac{\gamma^a - 1}{\gamma - 1}x + \gamma^a\frac{\tau^b - 1}{\tau - 1}y\right] \\
        & & \qquad + \tgamma^a\left\{s^{(pr)}\left[\frac{\tau^b - 1}{\tau - 1}(\varphi - 1)y\right] - \frac{\tau^b - 1}{\tau - 1}y\right\} \\
        & = & s^{(pr)}\gamma^a \frac{\tau^b - 1}{\tau - 1}y_2 - \tgamma^a s^{(pr)}\frac{\tau^b - 1}{\tau - 1}y_2.
    \end{eqnarray*}
\end{itemize}

This proves our claim. 
As before, let  $z_1,z_2$ be the elements obtained using Proposition \ref{H2 fixed under Hinfty} for cocyles $c_1$, $c_2$ respectively. 
Then the difference $z_1-z_2$ is 
\begin{eqnarray*}
    z_1 - z_2 & = & - \partial f(\tgamma, \tgamma^{-1}\ttau\tgamma)+ \partial f(\ttau, \tgamma)  - (\ttau - 1)s^{(pr)}(\varphi - 1)x + \tgamma s^{(pr)}\left[\frac{\tau^{\chi(\gamma)^{-1}} - 1}{\tau - 1}(\varphi - 1)y\right] - s^{(pr)}(\varphi - 1)y \\
    & = & f(\ttau) + \ttau f(\tgamma) - f(\tgamma) - \tgamma f(\tgamma^{-1}\ttau \tgamma) - (\ttau - 1)s^{(pr)}(\varphi - 1)x \\
    & & \quad + \tgamma s^{(pr)}\left[\frac{\tau^{\chi(\gamma)^{-1}} - 1}{\tau - 1}(\varphi - 1)y\right] - s^{(pr)}(\varphi - 1)y \\
    & = & s^{(pr)}(\varphi - 1)y - y + (\ttau - 1)\left[s^{(pr)}(\varphi - 1)x - x\right] - \tgamma\left[s^{(pr)}\frac{\tau^{\chi(\gamma)^{-1}}-1}{\tau - 1}(\varphi - 1)y - \frac{\tau^{\chi(\gamma)^{-1}}-1}{\tau - 1}y\right] \\
    & & \quad - (\ttau - 1)s^{(pr)}(\varphi - 1)x + \tgamma s^{(pr)}\left[\frac{\tau^{\chi(\gamma)^{-1}} - 1}{\tau - 1}(\varphi - 1)y\right] - s^{(pr)}(\varphi - 1)y \\
    & = & -(\tau - 1)x - (1 - \delta^{-1}\gamma)y.
\end{eqnarray*}
Therefore,
\[
    (-x_2, -y_2, -z_2) - (-x_1, -y_1, -z_1) = ((\varphi - 1)x, (\varphi - 1)y, -(\tau - 1)x - (1 - \delta^{-1}\gamma)y) = -d_1(0, x, y) \in \im d_1.
\]
This finishes the proof of Proposition \ref{H2 fixed under Hinfty} and hence Theorem \ref{Overconvergent theorem}.

\section*{Acknowledgements}
We thank Laurent Berger and Léo Poyeton and Kiran Kedlaya for interesting discussions. This work was supported by the National Research Foundation of Korea (NRF) grant funded by the Korea government (MSIT) (No. RS-2025-02262988 and No. RS-2025-00517685). The first author also gratefully acknowledges the support received from the HRI postdoctoral fellowship, where a part of this work was completed. The second author would like to thank Harish Chandra Research Institute for its hospitality during his visit to the institute where parts of this work were completed. The third author gratefully acknowledges support from Inspire Research Grant, DST, Govt. of India.

\printbibliography

\end{document}